\documentclass[a4paper,12pt, abstract=true]{scrartcl}
\usepackage{german}
\usepackage[ansinew]{inputenc}
\usepackage[german,english,french,USenglish]{babel}
\usepackage{latexsym,amssymb,amsfonts,bbm}
\usepackage{theorem}
\usepackage{xcolor}
\usepackage{graphicx}
\usepackage{esint}
\usepackage{stmaryrd}
\usepackage[normalem]{ulem}

\normalsize

\definecolor{darkgreen}{rgb}{.2,.8,0}

\newtheorem{dummytheorem}{Dummy-Theorem}[section]
\newcommand{\proofendsign}{$\Box$} 
\newtheorem{definition}[dummytheorem]{Definition}
\newtheorem{lemma}[dummytheorem]{Lemma}
\newtheorem{theorem}[dummytheorem]{Theorem}

\newtheorem{corollary}[dummytheorem]{Corollary}

\newtheorem{assumption}[dummytheorem]{Assumption}

\newenvironment{proof}{{\noindent \bf Proof }}
 {{\hspace*{\fill}\proofendsign\par\bigskip}}

\theorembodyfont{\normalfont}
\newtheorem{remarknorm}[dummytheorem]{Remark}
\newtheorem{examplenorm}[dummytheorem]{Example}

\newcommand{\V}{\mathbf{V}}
\newcommand{\N}{\mathbb{N}}
\newcommand{\Z}{\mathbb{Z}}

\newcommand{\R}{\mathbb{R}}
\newcommand{\C}{\mathbb{C}}

\newcommand{\F}{\mathbb{F}}
\newcommand{\D}{\mathbb{D}}
\newcommand{\B}{\mathbb{B}}

\newcommand{\pr}{\mathbb{P}}
\newcommand{\ex}{\mathbb{E}}

\newcommand{\covi}{\mathbb{C}{\rm ov}}

\newcommand{\eins}{\mathbbm{1}}

\newcommand{\vatr}{{\rm V@R}}
\newcommand{\avatr}{{\rm AV@R}}

\newcommand{\argmin}{{\rm argmin}}

\newcommand{\cadlag}{c\`adl\`ag}

\newcommand{\Had}{{\mbox{{\rm{\scriptsize{Had}}}}}}
\newcommand{\qHad}{{\mbox{{\rm{\scriptsize{qHad}}}}}}

\newcommand{\OFP}{(\Omega,{\cal F},\pr)}
\def\bcswitch{\left\{\renewcommand{\arraystretch}{1.2}\begin{array}{ccl}}
\def\ecswitch{\end{array}\right.}
\newcommand{\Timesir}{\raisebox{-2.2mm}{
\unitlength 1.00pt
\linethickness{0.5pt}
\begin{picture}(10.00,20.00)
\multiput(2.00,6.00)(0.12,0.18){50}{\line(0,1){0.18}}
\multiput(2.00,15.00)(0.12,-0.18){50}{\line(0,-1){0.18}}
\put(5.00,2.00){\makebox(0,0)[cc]{$\scriptstyle i=0$}}
\put(5.00,18.00){\makebox(0,0)[cc]{$\scriptstyle r$}}
\end{picture}
}}

\begin{document}

\title{Quasi-Hadamard differentiability of general risk functionals and its application}

\author{
Volker Krätschmer\footnote{Faculty of Mathematics, University of Duisburg--Essen, {\tt volker.kraetschmer@uni-due.de}}
\and {\setcounter{footnote}{3}}Alexander Schied\footnote{Department of Mathematics, University of Mannheim, {\tt schied@uni-mannheim.de}} {\setcounter{footnote}{6}}
\and
Henryk Z\"ahle\footnote{Department of Mathematics, Saarland University, {\tt zaehle@math.uni-sb.de}\hfill\break
A.S.~gratefully acknowledges support by Deutsche Forschungsgemeinschaft through the Research Training Group RTG 1953}}
\date{~
}

\maketitle

\begin{abstract}
We apply a suitable modification of the functional delta method to statistical functionals that arise from law-invariant coherent risk measures. To this end we establish   differentiability of the statistical functional in a relaxed Hadamard sense, namely with respect to a suitably chosen norm and in the directions of a specifically chosen  \lq\lq tangent space". We show that this notion of quasi-Hadamard differentiability  yields both strong laws and limit theorems for the asymptotic distribution of the plug-in estimators. Our results can be regarded  as  a contribution to the statistics and numerics of risk measurement and as a case study  for possible refinements of the functional delta method through fine-tuning the underlying notion of differentiability\end{abstract}



\newpage

\section{Introduction}\label{Introduction}

Let $X$ be a random variable describing the future profits and losses of a financial position.
When assessing its risk,   $\rho(X)$,   in terms of a risk measure $\rho$ it is   common  to estimate  $\rho(X)$ by means of a Monte Carlo procedure or from a sequence of historical data. This problem is well-posed when  $\rho$ is \lq\lq law-invariant" in the sense that there exists a functional $\mathcal R_\rho$ such that  $\rho(X)=\mathcal R_\rho(F_X)$  for $F_X$ denoting the distribution function  of $X$. In this case,   a natural estimate for $\rho(X)$ is given by $\mathcal R_\rho(\widehat F_n)$, where $\widehat F_n$ is the  empirical distribution function of the given data or another suitable estimate for $F_X$. In recent years, the statistical properties of the plug-in estimator $\mathcal R_\rho(\widehat F_n)$ and of the statistical functional $\mathcal R_\rho(\cdot)$ have been the subject of a number of studies. This includes studies on consistency and robustness \cite{ContDeguestScandolo,Kraetschmeretal2013}, elicitability \cite{Gneiting,Ziegel}, and weak limit theorems \cite{BelomestnyKraetschmer2012,BeutnerZaehle2010,PflugWozabal2010}. Here we continue and extend the latter class of studies by suitably adapting the modified functional delta method from \cite{BeutnerZaehle2010} to the case of a general law-invariant coherent risk measure $\rho$.

The functional delta method allows to lift a functional weak limit theorem from the level of the input data to the level of the plug-in estimators; see  \cite{vonMises,Reeds1976,Fernholz1983,Gill1989} for early references. This means in our present context and under suitable technical assumptions that if  there are numbers $(r_n)_{n\in\mathbb N}$ such that $r_n(\widehat F_n-F_X)$ converges in law to some random variable $B^\circ$, then $r_n(\mathcal R_\rho(\widehat F_n)-\mathcal R_\rho( F_X))$ converges in law to the random variable $\dot{\mathcal R}_{\rho,F_X}(B^\circ)$, where $\dot{\mathcal R}_{\rho,F_X}(\cdot)$ is a suitable derivative of the statistical functional $\mathcal R_\rho$ at $F_X$.
The problem of establishing a weak limit theorem for the plug-in estimators $\mathcal R_\rho(\widehat F_n)$ is thus reduced to  the following two independent steps:\\
$\phantom{a}$\quad (a) showing a weak limit theorem for the sequence $(\widehat F_n)$;\\
$\phantom{a}$\quad (b) proving that $\mathcal R_\rho$ admits the  derivative $\dot{\mathcal R}_{\rho,F_X}(\cdot)$.\\
When carrying out this  program  it is crucial to fine-tune the notion of differentiability of the functional $\mathcal R_\rho$. On the one hand, this notion needs to be sufficiently strong so that in combination with step (a) it yields the desired weak limit theorem for $(\mathcal R_\rho(\widehat F_n))_{n\in\mathbb N}$. On the other hand,  the  derivative $\dot{\mathcal R}_{\rho,F_X}(\cdot)$ should exist for a wide class of risk measures $\rho$ and distribution functions $F_X$, which will only be the case when differentiability is understood in a sufficiently weak sense. The classical notion used in the literature is 
Hadamard differentiability, which is stronger than 
differentiability in the sense of G\^ateaux but weaker than 
Fr\'echet differentiability (cf.\ e.g.\ \cite{Shapiro1990,vanderVaartWellner1996}). But, as observed in \cite{BeutnerZaehle2010},  this notion is still too strong to be applied to some of the most common law-invariant coherent risk measures $\rho$. This includes in particular many distortion risk measures $\rho_g$ that are defined as follows for  concave,  nondecreasing 
distortion functions $g:[0,1]\to[0,1]$ satisfying $g(0) = 0$ and $g(1) = 1$,
\begin{equation}\label{distortion risk measure}
\rho_g(X)=\mathcal R_{\rho_g}(F_X)=\int_{-\infty}^0g(F_X(x))\,dx - \int_0^\infty(1-g(F_X(x))\,dx.
\end{equation}
Therefore, a relaxed notion of \emph{quasi-Hadamard differentiability} was proposed in \cite{BeutnerZaehle2010}. This notion differs from classical Hadamard differentiability mainly by taking derivatives only in the directions of a relatively small \lq\lq tangent space" equipped with a suitably weighted norm and by relying on convergence with respect to this weighted norm. In contrast to the existing literature on tangential 
Hadamard differentiability in the context of the functional delta method, no single distribution function will have a finite length w.r.t.\ the imposed (weighted) norm, meaning that no single distribution function will belong to the \lq\lq tangent space".

In this paper, our goal is to extend the functional delta method based on quasi-Hadamard differentiability to a general class of law-invariant coherent risk measures. It is known that any such risk measure $\rho$ admits the following Kusuoka representation,
\begin{equation}\label{kusuoka rep eq}
\rho(X)=\sup_{g\in\mathcal G}\rho_g(X),
\end{equation}
where $\mathcal G$ is a class of distortion functions and  $\rho_g$ denotes the distortion risk measure  (\ref{distortion risk measure}) for  $g\in\mathcal G$. The so-called expectiles risk measures are examples  for risk measures that are of the form (\ref{kusuoka rep eq}) but not of the form (\ref{distortion risk measure}); see  Example \ref{example rmboe}. The Kusuoka representation (\ref{kusuoka rep eq}) will also play an important role in our formulas for the quasi-Hadamard derivatives of the statistical functionals $\mathcal R_\rho$.

By analyzing quasi-Hadamard differentiability of general law-invariant coherent risk measures  we are pursuing two different objectives. On the one hand we are aiming to contribute to the asymptotic analysis of plug-in estimators in risk measurement, with a view toward statistical inference, Monte Carlo computation, and optimization. On the other hand, we wish to provide a case study for possible refinements of the functional delta method through fine-tuning the underlying notion of differentiability.
In this latter respect, the scope of our analysis is not limited to applications in risk assessment.

Our main results on the quasi-Hadamard differentiability of statistical functionals $\mathcal R_\rho$ will be given in Theorems \ref{theorem qhd of r} and \ref{special case drm} under two different sets of assumptions. These assumptions will be illustrated in Sections \ref{Illustration section 2.1} and \ref{Illustration section 2.2} by means of a number of examples. Our applications to statistical inference are given in Section \ref{applications to statistics}. Specifically, our weak limit theorem for the sequence of plug-in estimator will be stated in Section \ref{asymptotic distribution of plug-in estimators} along with discussions for the cases of independent, weakly dependent, or strongly dependent data. In Section \ref{strong laws} a strong law of the form $ r_n({\cal R}_\rho(\widehat F_n)-{\cal R}_\rho(F_0))\to 0$ will be stated and discussed. The proofs of our main results are contained in Section \ref{Proof of QHD of R}. In Appendix \ref{appendix QHD and FDM} we recall the notion of quasi-Hadamard differentiability  and state some  auxiliary results in a general setting. Appendix \ref{Separability of uniform metric} contains an auxiliary result on Skorohod spaces.

\section{Main results}\label{Main result}

Let $(\Omega,{\cal F},\pr)$ be an atomless probability space and denote by $L^p(\Omega,{\cal F},\pr)$ the usual $L^p$-spaces for $p\in[1,\infty]$. Let ${\cal X}$ be a subspace of $L^{1}\OFP$ containing $L^{\infty}\OFP$. An element $X$ of ${\cal X}$ will be interpreted as the P\&L of a financial position. As usual (cf.\ e.g.\ \cite{ArtznerEtAl1999,FoellmerSchied2011}), we will say that a map $\rho:{\cal X}\to\R$ is a {\em coherent risk measure} if it is
\begin{itemize}
    \item monotone: $\rho(X)\ge\rho(Y)$ for all $X,Y\in{\cal X}$ with $X\le Y$,
    \item cash-invariant: $\rho(X+m)=\rho(X)-m$ for all $X\in{\cal X}$ and $m\in\R$,
    \item subadditive: $\rho(X+Y)\le\rho(X)+\rho(Y)$ for all $X,Y\in{\cal X}$,
    \item positively homogenous: $\rho(\lambda X)=\lambda\,\rho(X)$ for all $X\in{\cal X}$ and $\lambda\ge 0$.
\end{itemize}
A coherent risk measure $\rho$ will be called {\em law-invariant} if $\rho(X)=\rho(Y)$ whenever $X$ and $Y$ have the same law under $\pr$.

To give a typical example, let $g:[0,1]\to[0,1]$ be a concave distortion function, i.e.\ a concave and nondecreasing function with $g(0)=0$ and $g(1)=1$. The {\em distortion risk measure} associated with $g$ is defined by
\begin{equation}\label{def distortion risk measure - 0}
    \rho_g(X)\,:=\,\int_{-\infty}^0 g(F_X(x))\,dx-\int_0^\infty \big(1-g(F_X(x))\big)\,dx
\end{equation}
for every random variable $X\in L^0(\Omega,{\cal F},\pr)$ satisfying $\int_0^\infty g(1-F_{|X|}(x)\big)\,dx<\infty$, where $F_X$ and $F_{|X|}$ denote the distribution functions of $X$ and $|X|$, respectively. The set ${\cal X}$ of all such random variables forms a linear subspace of $L^1(\Omega,{\cal F},\pr)$; this follows from \ \cite[Proposition 9.5]{Denneberg1994} and \cite[Proposition 4.75]{FoellmerSchied2011}. It is known that $\rho_g$ is a law-invariant coherent risk measure; see, for instance, \cite{WangDhaene1998}.
If specifically $g(t)=(t/\alpha)\wedge 1$ for any fixed $\alpha\in(0,1)$, then we have ${\cal X}=L^1(\Omega,{\cal F},\pr)$ and $\rho_g$ is nothing but the Average Value at Risk at level $\alpha$. The latter is defined by
\begin{equation}\label{def avatr}
    \avatr_\alpha(X)\,:=\,\frac{1}{\alpha}\int_0^\alpha\vatr_s(X)\,ds,\qquad X\in L^1(\Omega,{\cal F},\pr),
\end{equation}
where $\vatr_s(X):=-F_X^\rightarrow(s)$ is the Value at Risk at level $s$. Here and elsewhere $F^{\leftarrow}(s):=\inf\{x\in\R:F(x)\ge s\}$ and $F^{\rightarrow}(s):=\inf\{x\in\R:F(x)>s\}$ denote the left-continuous and the right-continuous inverses of $F$ at $s$, respectively.

For any coherent risk measure $\rho:{\cal X}\to\R$ we can define a function $g_\rho:[0,1]\to[0,1]$ by
\begin{equation}\label{definition of g rho}
    g_\rho(t)\,:=\,\rho(-B_{1,t}),\qquad t\in[0,1], 
\end{equation}
where $B_{1,t}$ is a Bernoulli random variable with expectation $t.$ Clearly, $g_\rho$ is distortion function, i.e.\ a nondecreasing function $g_\rho:[0,1]\to[0,1]$ with $g_\rho(0)=0$ and $g_\rho(1)=1$, and we will refer to $g_\rho$ as the {\em distortion function associated with $\rho$}. For an alternative representation of $g_\rho$ see part (ii) of Theorem \ref{theorem kusuoka} below.

Distortion risk measures w.r.t.\ concave distortion functions are known to be building blocks of general law-invariant coherent risk measures on $L^{\infty}\OFP$. This is the so called Kusuoka representation (cf. \cite{Kusuoka}). More precisely, there exists some set ${\cal G}$ of concave distortion functions such that $\rho = \sup_{g\in {\cal G}}\rho_{g}.$ Recently, the Kusuoka representation has been extended to law-invariant coherent risk measures on more general spaces ${\cal X}$ under the following technical conditions; cf. \cite{BelomestnyKraetschmer2012} or \cite{KraetschmerZaehle2011}.

\begin{assumption}\label{assumption for kusuoka}
Let ${\cal X}$ be a Stonean vector lattice, i.e.\ $X\wedge Y,\,X\vee Y\in{\cal X}$ for all $X,Y\in{\cal X}$. Moreover, let $\rho$ be a law-invariant coherent risk measure on ${\cal X}$ and assume that the following conditions hold:
\begin{itemize}
    \item[(a)] $\lim_{k\to\infty}\rho(-(X-k)^{+})=0$ for all nonnegative $X\in {\cal X}$.
    \item[(b)] $\lim_{t\downarrow 0}g_\rho(t)=0$.
\end{itemize}
\end{assumption}

Remark \ref{remark on assumption a b} and Examples \ref{example ospm}--\ref{example hrm} below will illustrate Assumption \ref{assumption for kusuoka}. The next theorem deals with the above-mentioned extension of the Kusuoka representation for law-invariant coherent risk measures on general spaces ${\cal X}.$ Under Assumption \ref{assumption for kusuoka}, the representing concave distortion functions satisfy some additional useful properties.

\begin{theorem}\label{theorem kusuoka}
Suppose that Assumption \ref{assumption for kusuoka} holds. Then we can find some set ${\cal G}_\rho$ of continuous concave distortion functions such that the following assertions hold:
\begin{itemize}
    \item[(i)] ${\cal G}_\rho$ is compact w.r.t.\ the uniform metric.
    \item[(ii)] $g_\rho=\sup_{g\in{\cal G}_\rho}g$.
    \item[(iii)] $\sup_{g\in{\cal G}_\rho}g'(t)\le g_\rho(\gamma t)/(\gamma t)$ for all $\gamma,t\in(0,1)$.
    \item[(iv)] $\rho(X)=\sup_{g\in{\cal G}_\rho}\rho_{g}(X)$ for all $X\in{\cal X}$.
    \item[(v)] For every $X\in{\cal X}$, the mapping ${\cal G}_{\rho}\to\R$, $g\mapsto\rho_{g}(X)$ is lower semicontinuous w.r.t.\ the uniform metric, and it is even continuous if in addition $\int_{-\infty}^{0}g_{\rho}(\gamma F_{X}(x))\,dx < \infty$ holds for some $\gamma\in (0,1)$.
\end{itemize}
\end{theorem}

The proof of Theorem \ref{theorem kusuoka} can be found in Section \ref{proof of theorem kusuoka}. If $\rho$ is a distortion risk measure associated with a continuous concave distortion function $g$, then, of course, $g_\rho=g$ and ${\cal G}_\rho$ reduces to the singleton ${\cal G}_{\rho}=\{g\}$.

Now, let $\phi:\R\to[1,\infty)$ be a {\em weight function}, i.e.\ a continuous function that is nondecreasing on $(-\infty,0]$ and nonincreasing on $[0,\infty)$. Let $D_{\phi}$ be the space of all \cadlag\ functions $v$ on $\R$ with $\|v\|_\phi:=\|v\phi\|_\infty<\infty$, 
where $\|f\|_\infty:=\sup_{x\in\R}|f(x)|$ denotes the sup-norm of a real-valued function $f$ on $\R$. Let $\F_{\cal X}$ be the set of distribution functions of the random variables from ${\cal X}$. For any given $F_0\in\F_{\cal X}$, we denote by $D_{\phi,F_0}$ the subspace of all $v\in D_\phi$ vanishing outside $[F_0^{\rightarrow}(0),F_0^{\leftarrow}(1)]$, and by $C_{\phi,F_0}$ the subspace of all functions in $D_{\phi,F_0}$ whose discontinuity points on $(F_0^{\rightarrow}(0),F_0^{\leftarrow}(1))$ are discontinuity points of $F_{0}$. Both $C_{\phi,F_0}$ and $D_{\phi,F_0}$ will be equipped with the norm $\|\cdot\|_{\phi}$.

\begin{assumption}\label{assumption for qhd}
Let $\rho$ be a law-invariant coherent risk measure on ${\cal X}$, and $g_\rho$ the distortion function associated with $\rho$ as defined in (\ref{definition of g rho}). Let $F_0\in\F_{\cal X}$, $\phi$ be a weight function as above, and assume that the following conditions holds:
\begin{itemize}
    \item[(a)] There exists a finite set $D(F_0)$of real numbers such that $F_0$ is continuously differentiable with strictly positive derivative on $(F_0^{\rightarrow}(0),F_0^{\leftarrow}(1))\setminus D(F_{0})$.
    \item[(b)] $\int_{F_0^\rightarrow(0)}^{F_0^\leftarrow(1)} g_\rho(\gamma F_0(x))/(F_0(x)\phi(x))\,dx<\infty$ for some $\gamma\in(0,1)$.
       \end{itemize}
\end{assumption}

Remarks \ref{remark on assumption d} and \ref{example for assumption d} below will illustrate Assumption \ref{assumption for qhd}\,(b). When $\rho$ is law-invariant we may associate with $\rho$ a statistical functional ${\cal R}_\rho$ defined on the class  $\F_{\cal X}$ of all distribution functions of  elements of ${\cal X}$ via
\begin{equation}\label{definition of risk functional}
    {\cal R}_\rho(F_X)\,:=\,\rho(X),\qquad X\in\cal X.
\end{equation}
The following theorem involves the notion of quasi-Hadamard differentiability, which is recalled in Definition \ref{definition quasi hadamard} in the Appendix \ref{appendix QHD and FDM}. In our specific setting, the roles of $\V$, $\V_0$, $\C_0$ and $\V'$ from Definition \ref{definition quasi hadamard} are played by $D$, $D_{\phi,F_0}$, $C_{\phi,F_0}$ and $\R$, respectively. Here $D$ is the space of all bounded \cadlag\ functions. That is, quasi-Hadamard differentiability of ${\cal R}_\rho$ at $F_0$ tangentially to $C_{\phi,F_0}\langle D_{\phi,F_0}\rangle$ means that we can find some continuous map $\dot{\cal R}_{\rho,F_0}:C_{\phi,F_0}\to\R$ such that
$$
    \lim_{n\to\infty}\,\Big|\frac{{\cal R}_\rho(F_0+h_nv_n)-{\cal R}_\rho(F_0)}{h_n}-\dot{\cal R}_{\rho,F_0}(v)\Big|\,=\,0
$$
holds for every triplet $(v,(v_n),(h_n))$ with $v\in C_{\phi,F_0}$, $(v_n)\subset D_{\phi,F_0}$ satisfying $\|v_n-v\|_{\phi}\to 0$ and $(F_0+h_nv_n)\subset\F_{\cal X}$, and $(h_n)\subset(0,\infty)$ satisfying $h_n\to 0$.

\begin{theorem}\label{theorem qhd of r}
Suppose that Assumptions \ref{assumption for kusuoka} and \ref{assumption for qhd} hold, and let ${\cal G}_\rho$ be given by Theorem \ref{theorem kusuoka}. Then the functional ${\cal R}_\rho$ defined in (\ref{definition of risk functional}) is quasi-Hadamard differentiable at $F_0$ tangentially to $C_{\phi,F_0}\langle D_{\phi,F_0}\rangle$ with quasi-Hadamard derivative $\dot {\cal R}_{\rho,F_0}$ given by
\begin{equation}\label{gh ableitung}
    \dot{\cal R}_{\rho,F_0}(v)\,:=\,\lim_{\varepsilon\downarrow 0}\,\sup_{g\in{\cal G}_\rho(F_0,\varepsilon)}\int_{F_0^\rightarrow(0)}^{F_0^\leftarrow(1)} g'(F_0(x))\,v(x)\,dx\,,\qquad v\in C_{\phi,F_0},
\end{equation}
where ${\cal G}_{\rho}(F_0,\varepsilon)$ denotes the set of all $g\in {\cal G}_{\rho}$ satisfying ${\cal R}_{\rho}(F_0) - \varepsilon\leq {\cal R}_{g}(F_0)$, and $g'$ stands for the right-sided derivative of $g$.

If in addition $\int_{-\infty}^{0} g_{\rho}(\delta F_{0}(x))\,dx < \infty$ holds for some $\delta\in (0,1)$, then the set ${\cal G}_\rho(F_{0})$ of all $g\in{\cal G}_\rho$ satisfying ${\cal R}_\rho(F_0) = {\cal R}_g(F_0)$ is nonempty, and the quasi-Hadamard derivative has the following form
\begin{equation}\label{def of qh ableitung}
    \dot{\cal R}_{\rho,F_0}(v)\,:=\,\sup_{g\in{\cal G}_\rho(F_{0})}\int_{F_0^\rightarrow(0)}^{F_0^\leftarrow(1)} g'(F_0(x))\,v(x)\,dx\,,\qquad v\in C_{\phi,F_0}.
\end{equation}

\end{theorem}

The proof of Theorem \ref{theorem qhd of r} can be found in Section \ref{ersterHauptbeweis}.

\begin{remarknorm}
Under condition (b) from Assumptions \ref{assumption for qhd}, we obtain $\int_{-\infty}^{0}g_{\rho}(\delta F_{0}(x))\,dx < \infty$ for some $\delta\in (0,1)$ if $F_{0}\,\phi$ is bounded on $(-\infty,0)$. Indeed, choosing $\gamma\in (0,1)$ from part (b) of Assumptions \ref{assumption for qhd},
we may conclude with H\"older's inequality that the function
$g_{\rho}(\gamma F_0)\eins_{(F_{0}^{\rightarrow}(0)\wedge 0,0)}=((g_{\rho}(\gamma F_0))/(F_0\,\phi))F_0\,\phi\eins_{(F^{\rightarrow}(0)\wedge 0,0)}$ is integrable w.r.t.\ the Lebesgue measure on $\R$. {\hspace*{\fill}$\Diamond$\par\bigskip}
\end{remarknorm}

For distortion risk measures $\rho=\rho_g$ with continuous distortion function $g$, Theorem \ref{theorem qhd of r} can be improved and the assumptions on $\phi$ and $F_0$ can be relaxed as follows. In this case we clearly have  $g_\rho=g$.

\begin{assumption}\label{assumption for qhd - distortion}
Let $\rho_{g}:{\cal X}\rightarrow\R$ be a distortion risk measure associated with a continuous concave distortion function $g$ as defined in (\ref{def distortion risk measure - 0}). Moreover, let $F_0\in\F_{{\cal X}}$, $\phi$ be a weight function, and assume that the following conditions holds:
\begin{itemize}
    \item[(a)] The set of points $x\in(F_{0}^{\rightarrow}(0),F_{0}^{\leftarrow}(1))$ for which $g$ is not differentiable at $F_{0}(x)$ has Lebesgue measure zero.
    \item[(b)] Assumption \ref{assumption for qhd}\,(b) holds for $\rho=\rho_g$.
\end{itemize}
\end{assumption}

\begin{theorem}\label{special case drm}
Suppose that Assumption \ref{assumption for qhd - distortion} holds. Then the functional ${\cal R}_g={\cal R}_{\rho_g}$ is quasi-Hadamard differentiable at $F_0$ tangentially to $C_{\phi,F_0}\langle D_{\phi,F_0}\rangle$ with quasi-Hadamard derivative $\dot {\cal R}_{g}$ given by
\begin{equation}\label{special case drm - eq}
   \dot{\cal R}_{g}(v)\,:=\,\int_{F_0^\rightarrow(0)}^{F_0^\leftarrow(1)} g'(F_0(x))\,v(x)\,dx,\qquad v\in C_{\phi,F_0},
\end{equation}
where $g'$ denotes as before the right-sided derivative of $g$.
\end{theorem}

The proof of Theorem \ref{special case drm} can be found in Section \ref{zweiterHauptbeweis}. Theorem \ref{special case drm} partially generalizes Theorem 2.2 in \cite{BeutnerZaehle2010} where it was assumed that $g'$ is bounded. On the other hand, in \cite{BeutnerZaehle2010} the distortion function $g$ was not required to be concave.


\subsection{Illustration of Assumption \ref{assumption for kusuoka}}\label{Illustration section 2.1}

\begin{remarknorm}\label{remark on assumption a b}
It is worth pointing out that conditions (a)--(b) in Assumption \ref{assumption for kusuoka} are always fulfilled if one can find a 
complete norm $\|\cdot\|$ on the Stonean vector lattice ${\cal X}$ such that, for every $X,Y,X_1,X_2,\ldots\in{\cal X}$, we have $\|X\|\le\|Y\|$ when $|X|\le|Y|$, and $\lim_{k\to\infty}\|X_k\|=0$ when $X_k\uparrow 0$ $\pr$-a.s. As already discussed in \cite[Remark 3.2]{BelomestnyKraetschmer2012} this follows from results in \cite{RuszczynskiShapiro2006}. General classes of random variables meeting these requirements are $L^p(\Omega,{\cal F},\pr)$ equipped with the usual 
$L^p$-norm for $p\in[1,\infty]$, and, more generally, the Orlicz heart $H^\psi(\Omega,{\cal F},\pr)$ equipped with the Luxemburg 
norm associated with a continuous Young function $\psi$; for more details see \cite[Remark 3.2]{BelomestnyKraetschmer2012}.
{\hspace*{\fill}$\Diamond$\par\bigskip}
\end{remarknorm}

The discussion in the preceding Remark \ref{remark on assumption a b} shows that every law-invariant coherent risk measure $\rho$ which is defined on some $L^p(\Omega,{\cal F},\pr)$ for some $p\in[1,\infty]$, or, more generally, on some Orlicz heart $H^\psi(\Omega,{\cal F},\pr)$ for some continuous Young function $\psi$, is covered by Theorem \ref{theorem qhd of r}. Examples are the one-sided $p$th moment risk measure defined on $L^p(\Omega,{\cal F},\pr)$, $p\in[1,\infty)$, the expectiles-based risk measure defined on $L^2(\Omega,{\cal F},\pr)$, and the Haezendonck--Goovaerts risk measure associated with a continuous Young function $\psi$ defined on $H^\psi(\Omega,{\cal F},\pr)$. For details see the following Examples \ref{example ospm}--\ref{example hrm} and \cite[Examples 3.2--3.4]{BelomestnyKraetschmer2012}.

\begin{examplenorm}\label{example ospm}
Given $a\in(0,1]$ and $p\in[1,\infty)$, the {\em one-sided $p$th moment risk measure} is defined by
\begin{eqnarray*}
     \rho(X)\,:=\,-\ex[X]+a\| (X-\ex[X])^{-} \|_{p}, \qquad X\in L^p(\Omega,{\cal F},\pr).
\end{eqnarray*}
In \cite{Fischer2003} it has been shown that $\rho$ is a law-invariant coherent risk measure. The associated distortion function is given by $g_{\rho}(t)=t+a(1-t)t^{1/p}$, $t\in[0,1]$. Since $\rho$ is defined on $L^p(\Omega,{\cal F},\pr)$, it satisfies conditions (a)--(b) in Assumption \ref{assumption for kusuoka}; cf.\ Remark \ref{remark on assumption a b}. But by Lemma A.5 in \cite{KraetschmerZaehle2011} it not a distortion risk measure for $a>0$.
{\hspace*{\fill}$\Diamond$\par\bigskip}
\end{examplenorm}

\begin{examplenorm}\label{example rmboe}
In \cite{BelliniKlarMuellerRosazza-Gianin2013} it has been pointed out that expectiles, genuinely introduced in \cite{NeweyPowell1987}, may be viewed as law-invariant coherent risk measures. The {\it expectiles-based risk measure} associated with $\alpha\in [1/2,1)$ is defined by
$$
     \rho(X):=\,\argmin\Big\{(1-\alpha)\| ((-X)-x)^{-}\|_{2}^{2}\,+\,\alpha \| ((-X)-x)^{+} \|_{2}^{2}~:~x\in\R\Big\},~\, X\in L^{2}(\Omega,{\cal F},\pr).
$$
The associated distortion function is given by $g_{\rho}(t)=(\alpha t)/(1-\alpha+t(2\alpha-1))$, $t\in[0,1]$. Since $\rho$ is defined on $L^2(\Omega,{\cal F},\pr)$, it satisfies conditions (a)--(b) in Assumption \ref{assumption for kusuoka}; cf.\ Remark \ref{remark on assumption a b}.
The set $\mathcal{G}_\rho$ corresponding to the Kusuoka representation of $\rho$ is identified in  \cite[Theorem 8]{Delbaen2013}. This result implies  in particular that $\rho$ is not a distortion risk measure unless $\alpha=1/2$.
{\hspace*{\fill}$\Diamond$\par\bigskip}
\end{examplenorm}

\begin{examplenorm}\label{example hrm}
Let $\psi$ be a strictly increasing continuous Young function with $\psi(1)=1$. By Young function we mean a nondecreasing, unbounded, convex function $\psi:\R_+\to\R_+$ with $\psi(0+)=0$. Let $\mathcal{X}=H^\psi(\Omega,{\cal F},\pr)$ be the Orlicz heart associated with $\psi$,
$$
    H^\psi(\Omega,{\cal F},\pr)\,:=\,\big\{X\in L^0\,:\,\ex[\,\Psi(c|X|)\,]<\infty\mbox{ for all $c>0$}\big\},
$$
and fix $\alpha\in(0,1)$. It was shown in \cite{Goovaertsetal2004} that for every $X\in H^\psi(\Omega,{\cal F},\pr)$ and every $x\in\R$ with $\pr[X>x]>0$ there exists a unique real number $\pi_{\alpha}^{\psi}(X,x)>x$ such that
$$
     \ex\Big[\psi\Big(\frac{(X-x)^{+}}{\pi_{\alpha}^{\psi}(X,x) - x}\Big)\Big] = 1- \alpha.
$$
Therefore we may uniquely define
$$
     \rho(X)\,:=\,\inf\big\{\pi_{\alpha}^{\psi}(- X,x)\,:\,x\in\R\mbox{ with } \pr[-X>x]>0\big\},\quad X\in H^\psi(\Omega,{\cal F},\pr).
$$
It is known from \cite{BelliniRosazzaGianin2008,KraetschmerZaehle2011} that $\rho$ is a law-invariant coherent risk measure, sometimes referred to as {\em  Haezendonck--Goovaerts risk measure} associated with $\psi$ and $\alpha$. For the associated distortion function we have
\begin{equation}\label{example hrm - eq 1}
     g_{\rho}(t)\,\le\,1\wedge\Big(t+\frac{1-t}{\psi^{-1}((1-\alpha)/t)}\Big)\qquad\mbox{for all }t\in[0,1],
\end{equation}
where $\psi^{-1}$ denotes the inverse function of $\psi$. Since $\rho$ is defined on $H^\psi(\Omega,{\cal F},\pr)$, it satisfies conditions (a)--(b) in Assumption \ref{assumption for kusuoka}; cf.\ Remark \ref{remark on assumption a b}. However, as explained in the Appendix of \cite{KraetschmerZaehle2011}, $\rho$ is not a distortion risk measure in general.
{\hspace*{\fill}$\Diamond$\par\bigskip}
\end{examplenorm}


\subsection{Illustration of Assumption \ref{assumption for qhd}\,(b)}\label{Illustration section 2.2}

\begin{remarknorm}\label{remark on assumption d}
If $-\infty<F_0^{\rightarrow}(0)$ and $F_0^{\leftarrow}(1)<\infty$, i.e.\ if $dF_0$ has compact support, then the integrability condition (b) in Assumption \ref{assumption for qhd} is fulfilled for every weight function $\phi$. That is, in this case we may assume without loss of generality that $\phi=\eins$.
{\hspace*{\fill}$\Diamond$\par\bigskip}
\end{remarknorm}

\begin{remarknorm}\label{example for assumption d}
Assume that $-\infty=F_0^{\rightarrow}(0)$ and $F_0^{\leftarrow}(1)=\infty$. If $\limsup_{t\to 0+}g_{\rho}(t)/t^{\beta} < \infty$, which equivalently means that $g_\rho(t)\le C\,t^\beta$, $t\in[0,1]$, for some constants $C\in(0,\infty)$ and $\beta\in(0,1]$, then the integrability condition (b) in Assumption \ref{assumption for qhd} is implied by
\begin{equation}\label{example for assumption d - eq}
    \int_{-\infty}^\infty\frac{1}{F_0(x)^{1-\beta}\phi(x)}\,dx\,<\,\infty.
\end{equation}
Conditions (b) in Assumption \ref{assumption for qhd} and the integrability condition (\ref{example for assumption d - eq}) are even equivalent if in addition $\liminf_{t\to 0+}g_{\rho}(t)/t^{\beta}>0$ holds. In particular, for every weight function $\phi$, Assumption \ref{assumption for qhd}\,(b)  and (\ref{example for assumption d - eq}) are equivalent if $\lim_{t\to 0+}g_{\rho}(t)/t^{\beta}\in (0,\infty)$.

(i) For the one-sided $p$th moment risk measure defined in Example \ref{example ospm}, with $p\in[1,\infty)$, we have $\lim_{t\to 0+}g_{\rho}(t)/t^{1/p}\in (0,\infty)$. Thus, for every weight function $\phi$, Assumption \ref{assumption for qhd}\,(b) is equivalent to
$$
    \int_{-\infty}^\infty\frac{1}{F_0(x)^{(p-1)/p}\,\phi(x)}\,dx\,<\,\infty.
$$

(ii) For the expectiles-based risk measure as defined in Example \ref{example rmboe}, with $\alpha\in [1/2,1)$, we have $\lim_{t\to 0+}g_{\rho}(t)/t=\alpha/(1-\alpha)\in(0,\infty)$. Thus, for every weight function $\phi$, Assumption \ref{assumption for qhd}\,(b) is equivalent to
$$
    \int_{-\infty}^\infty\frac{1}{\phi(x)}\,dx\,<\,\infty.
$$

(iii) For the Haezendonck--Goovaerts risk measure as defined in Example \ref{example hrm}, with $\psi$ and $\alpha\in(0,1)$, we may conclude from (\ref{example hrm - eq 1}) that $\limsup_{t\to 0+}g_{\rho}(t)/t^{\beta}<\infty$ for some $\beta\in (0,1]$ whenever $\liminf_{t\to 0 +}\psi^{-1}((1-\alpha)/t) t^{\beta} > 0$. The latter condition may be described equivalently by the condition $\liminf_{t\to 0 +}\psi^{-1}(1/t) t^{\beta} > 0$, and is satisfied if $\limsup_{x\to\infty}\psi(x)/x^{1/\beta}<\infty$. Thus, condition (b) in Assumption \ref{assumption for qhd}  is satisfied if (\ref{example for assumption d - eq}) holds with this choice of $\beta$.

(iv) For $\rho=\avatr_\alpha$ defined in (\ref{def avatr}) we obviously have $g_\rho(t)=g(t)=(t/\alpha)\wedge 1$. So we have, in particular, $\lim_{t\to 0+}g_{\rho}(t)/t=1/\alpha\in(0,\infty)$. Thus, Assumption \ref{assumption for qhd}\,(b) is satisfied for every $F_0\in\mathbb{F}_{L^1(\Omega,\mathcal{F},\mathbb{P})}$ if and only if
$$
    \int_{-\infty}^\infty\frac{1}{\phi(x)}\,dx\,<\,\infty.
$$
{\hspace*{\fill}$\Diamond$\par\bigskip}
\end{remarknorm}


\section{Application to statistical inference}\label{applications to statistics}

The quasi-Hadamard derivative $\dot{\cal R}_{\rho,F_0}$ of ${\cal R}_\rho$ evaluated at $G-F_0$ can be seen as a measure for the sensitivity of a sequence of plug-in estimators for ${\cal R}_\rho(F_0)$ w.r.t.\ a contamination $F_{0,h}:=(1-h)F_0+hG$ of $F_0$ (with $h$ small) for some given distribution function $G$; cf.\ the discussion in Section 5 of \cite{Kraetschmeretal2012a}. It also facilitates the derivation of weak and strong limit theorems for plug-in estimators of ${\cal R}_\rho(F_0)$, and establishing such limit theorems will be our goal in this section. For any given $F_0\in\F_{\cal X}$, we equip $D_{\phi,F_0}$ with the trace $\sigma$-algebra ${\cal D}_{\phi,F_0}:={\cal D}\cap D_{\phi,F_0}$, where ${\cal D}$ denotes the $\sigma$-algebra generated by the coordinate projections on the space $D$ of all bounded \cadlag\ functions on $\R$. Notice that ${\cal D}_{\phi,F_0}$ coincides with the ball $\sigma$-algebra on $(D_{\phi,F_0},\|\cdot\|_\phi)$. This fact may be obtained for $\phi=\eins$ following \cite[Problem IV.2.2]{Pollard1984}. For general $\phi$, one uses that $D_{\phi,F_{0}}$ and $D_{\eins,F_{0}}$ are isometrically isomorphic for their respective norms. As a consequence, any $\|\cdot\|_{\phi}$-closed and separable subset of $D_{\phi,F_{0}}$ belongs to ${\cal D}_{\phi,F_0}$; see \cite[hint for Problem 1.7.4]{vanderVaartWellner1996}.
Convergence in distribution will be understood in the sense of \cite{Pollard1984,ShorackWellner1986}.


\subsection{Asymptotic distributions of plug-in estimators}\label{asymptotic distribution of plug-in estimators}

The quasi-Hadamard differentiability established in Theorems \ref{theorem qhd of r} and \ref{special case drm} provides a very general device to determine asymptotic distributions of plug-in estimators of law-invariant coherent risk measures. Recall that the probability space $(\Omega,{\cal F},\pr)$ was assumed to be atomless.

\begin{theorem}\label{main theorem coroll}
Suppose that Assumptions \ref{assumption for kusuoka} and \ref{assumption for qhd} hold, or that Assumption \ref{assumption for qhd - distortion} holds. Let
$\widehat F_n:\Omega\to D$ be a map for every $n\in\N$, and assume that the following conditions hold:
\begin{itemize}
    \item[(a)] $\widehat F_n$ takes values only in $\F_{\cal X}$ and is $({\cal F},{\cal D})$ measurable, $n\in\N$.
    \item[(b)] $\widehat F_n-F_0$ takes values only in $D_{\phi,F_0}$, $n\in\N$.
    \item[(c)] There are some random element $B^\circ$ of $(D_{\phi,F_0},{\cal D}_{\phi,F_0})$ as well as some $\|\cdot\|_{\phi}$-separable and ${\cal D}_{\phi,F_0}$-measurable subset $C\subset C_{\phi,F_0}$ with $\pr[B^{\circ}\in C]=1$, and a nondecreasing sequence $(r_n)\subset(0,\infty)$ with $r_n\uparrow\infty$, such that
    \begin{equation}\label{condition on B circ}
        r_n(\widehat F_n-F_0)\stackrel{\sf d}{\longrightarrow}B^\circ\qquad\mbox{in $(D_{\phi,F_0},{\cal D}_{\phi,F_0},\|\cdot\|_{\phi})$}.
    \end{equation}
\end{itemize}
Then
\begin{equation}\label{conv of emp process statistics}
    r_n({\cal R}_\rho(\widehat F_n)-{\cal R}_\rho(F_0))\stackrel{\sf d}{\longrightarrow}\dot{\cal R}_{\rho,F_0}(B^\circ)\qquad\mbox{in $(\R,{\cal B}(\R))$},
\end{equation}
with $\dot{\cal R}_{\rho,F_0}$ defined as in (\ref{gh ableitung}). If in addition $\int_{-\infty}^{0} g_{\rho}(\delta F_{0}(x))\,dx < \infty$ holds for some $\delta\in (0,1)$ (which, e.g., is the case if the restriction of $F_0\,\phi$ to $(-\infty,0)$ is bounded), then $\dot{\cal R}_{\rho,F_0}$ is as in (\ref{def of qh ableitung}).
\end{theorem}

\begin{proof}
First of all we note that $\widehat F_n-F_0$ is $({\cal F},{\cal D}_{\phi,F_0})$-measurable, because we assumed that $\widehat F_n$ is $({\cal F},{\cal D})$-measurable, $\widehat F_n-F_0 \in D_{\phi,F_0}$ and ${\cal D}_{\phi,F_0}={\cal D}\cap D_{\phi,F_0}$. In particular, $r_n(\widehat F_n-F_0)$ is a random element of $(D_{\phi,F_0},{\cal D}_{\phi,F_0})$.

If Assumptions \ref{assumption for kusuoka} and \ref{assumption for qhd} are satisfied, then Theorem \ref{theorem qhd of r} yields that ${\cal R}_\rho$ is quasi-Hadamard differentiable at $F_0$ tangentially to $C_{\phi,F_0}\langle D_{\phi,F_0}\rangle$. If Assumption \ref{assumption for qhd - distortion} holds, then $\rho$ is a distortion risk measure and we can apply Theorem \ref{special case drm} to obtain the same conclusion as before. In particular, in both cases ${\cal R}_\rho$ is quasi-Hadamard differentiable at $F_0$ tangentially to $C\langle D_{\phi,F_0}\rangle$.
So the claim of Theorem \ref{main theorem coroll} would follow from the Modified Functional Delta-Method given in \cite[Theorem 4.1]{BeutnerZaehle2010} if we can show that condition (c) in Theorem 4.1 in  \cite{BeutnerZaehle2010} holds. The latter condition requires that $\omega'\mapsto{\cal R}_\rho(W(\omega')+F_0)$ is $({\cal F}',{\cal B}(\R))$-measurable whenever $W$ is a measurable map from some measurable space $(\Omega',{\cal F}')$ to $(D_{\phi,F_0},{\cal D}_{\phi,F_0})$ such that $W(\omega')+F_0\in\F_{\cal X}$ for all $\omega'\in\Omega'$. This condition ensures in particular that left-hand side in (\ref{conv of emp process statistics}) is a real-valued random variable.

Since $W$ is $({\cal F}',{\cal D}_{\phi,F_0})$-measurable and ${\cal D}_{\phi,F_0}$ is the projection $\sigma$-field, we obtain in particular $({\cal F}',{\cal B}(\R))$-measurability of $\omega'\mapsto W(x,\omega')$ for every $x\in\R$. Since $x\mapsto W(x,\omega')$ is right-continuous for every $\omega'$, the mapping $(x,\omega')\mapsto W(x,\omega')$ is ${\cal B}(\R)\otimes{\cal F}$-measurable. In particular, the same is true for the mapping $(x,\omega')\mapsto W(x,\omega')+F_0(x)$. In view of
$$
    {\cal R}_g(F)\,=\,\int_{-\infty}^0 g(F(x))\,dx-\int_0^\infty \big(1-g(F(x)\big)\,dx\qquad\mbox{for all }F\in\F_{\cal X},\,g\in{\cal G}_\rho,
$$
we may conclude from Fubini's theorem that the mapping $\omega\mapsto{\cal R}_g(W(\omega')+F_0)$ is $({\cal F}',{\cal B}(\R))$-measurable for every $g\in{\cal G}_\rho$. Moreover, for any $F\in\mathbb{F}_{{\cal X}},$ the mapping $g\mapsto {\cal R}_{g}(F)$ on ${\cal G}_{\rho}$ is lower semicontinuous w.r.t.\ the uniform metric by Theorem \ref{theorem kusuoka}\,(v). Recalling compactness of ${\cal G}_{\rho},$ there is some countable subset ${\cal G}_{0}\subset {\cal G}_{\rho}$ such that  ${\cal R}_{\rho}(F) = \sup_{g\in {\cal G}_{0}} {\cal R}_{g}(F)$ holds for every $F\in\mathbb{F}_{{\cal X}}$. Hence, the mapping $\omega'\mapsto{\cal R}_\rho(W(\omega')+F_0)$ is
$({\cal F}',{\cal B}(\R))$-measurable. This completes the proof.
\end{proof}

\begin{remarknorm}\label{remark on literature}
The preceding theorem is related as follows to previously obtained results in the literature. In the special case of a distortion risk measure associated with a (possibly nonconcave) distortion function $g$ whose right-sided derivative $g'$ is bounded, the result of Theorem \ref{main theorem coroll} also follows from Theorem 2.5 in \cite{BeutnerZaehle2010}. When, on the other hand, $\rho$ is a general law-invariant coherent risk measure and $\widehat F_n$ are the  empirical distribution functions of a
strongly mixing stationary sequences of random variables with exponential decay of the mixing coefficients, the asymptotic distributions of plug-in estimators  have been derived in Theorems 3.1 and 4.1 in \cite{BelomestnyKraetschmer2012} under certain integrability conditions.
{\hspace*{\fill}$\Diamond$\par\bigskip}
\end{remarknorm}

Notice that the subset $C_{0,\phi,F_0}\subset C_{\phi,F_0}$ of all functions $v\in C_{\phi,F_0}$ satisfying $\lim_{x\to\pm\infty}v(x)=0$ is $\|\cdot\|_{\phi}$-separable and ${\cal D}_{\phi,F_0}$-measurable, as required by condition (c) in Theorem \ref{main theorem coroll}. For details see Corollary \ref{monotone separable} in the Appendix \ref{Separability of uniform metric}. Of course, we have $C_{0,\phi,F_0}=C_{\phi,F_0}$ when $\lim_{x\to\pm\infty}\phi(x)=\infty$. The following Examples \ref{example iid data}--\ref{example strongly dependent data} illustrate condition (c) of Theorem \ref{main theorem coroll}, where $\widehat F_n$ will always be the empirical distribution function $\widehat F_n:=\frac{1}{n}\sum_{i=1}^n\eins_{[X_i,\infty)}$ of the first $n$ variables of a sequence $X_1,X_2,\ldots$ of identically distributed random variables on some probability space $(\Omega,{\cal F},\pr)$.

\begin{examplenorm}{\em (Independent data)}\label{example iid data}
Assume that $X_1,X_2,\dots$ are i.i.d.\ with distribution function $F_0$, and let $\phi$ be a weight function. If $\int\phi^2dF_0<\infty$, then Theorem 6.2.1 in \cite{ShorackWellner1986} shows that for the empirical distribution function $\widehat F_n$ of $X_1,\dots,X_n$
$$
    \sqrt{n}(\widehat F_n-F_0)\stackrel{\sf d}{\longrightarrow} B_{F_0}^\circ\qquad\mbox{(in $(D_{\phi,F_0},{\cal D}_{\phi,F_0},\|\cdot\|_\phi)$)},
$$
where $B_{F_0}^\circ$ is an $F_0$-Brownian bridge, i.e.\ a centered Gaussian process with covariance function $\Gamma(y_0,y_1)=F_0(y_0\wedge y_1)(1-F_0(y_0\vee y_1))$. Notice that $B_{F_0}^\circ$ jumps where $F_0$ jumps, and that $\lim_{x\to\pm\infty}B_{F_0}^\circ(x)=0$. Thus, $B_{F_0}^\circ$ takes values only in the $\|\cdot\|_{\phi}$-separable and ${\cal D}_{\phi,F_0}$-measurable subset $C_{0,\phi,F_0}$ of $C_{\phi,F_0}$.
{\hspace*{\fill}$\Diamond$\par\bigskip}
\end{examplenorm}

\begin{examplenorm}{\em (Weakly dependent data)}\label{example weakly dependent data}
Let $(X_i)$ be strictly stationary and $\alpha$-mixing with mixing coefficients satisfying $\alpha(n)=\mathcal{O}(n^{-\theta})$ for some $\theta>1+\sqrt{2}$. Let $F_0$ be the distribution function of the $X_i$, let $\lambda\ge 0$, and set $\phi_\lambda(x):=(1+|x|)^\lambda$. If $F_0$ is continuous and has a finite $\gamma$-moment for some $\gamma>2\theta\lambda/(\theta-1)$, then it can easily be deduced from Theorem 2.2 in \cite{ShaoYu1996} that for the empirical distribution function $\widehat F_n$ of $X_1,\dots,X_n$
$$
    \sqrt{n}(\widehat F_{n}-F_0)\stackrel{\sf d}{\longrightarrow}\widetilde B_{F_0}^\circ\qquad\mbox{(in $(D_{\phi_\lambda,F_0},{\cal D}_{\phi_\lambda,F_0},\|\cdot\|_{\phi_\lambda})$)}
$$
with $\widetilde B_F^\circ$ a continuous
centered Gaussian process with covariance function $\Gamma(y_0,y_1)=F_0(y_0\wedge y_1)(1-F_0(y_0\vee y_1))+\sum_{i=0}^1\sum_{k=2}^{\infty}\covi(\eins_{\{X_1 \le y_i\}}, \eins_{\{X_k \le y_{1-i}\}})$. See also \cite[Section 3.3]{BeutnerZaehle2010}. Notice that $B_{F_0}^\circ$ takes values only in the $\|\cdot\|_{\phi}$-separable and ${\cal D}_{\phi,F_0}$-measurable subset of all continuous functions within $D_{\phi,F_0}$. Also notice that under some mild regularity conditions, strictly stationary GARCH$(p,q)$ processes are $\alpha$-mixing with $\alpha(n)\le c\,\varrho^n$, $n \in\N$, for some constants $c>0$ and $\varrho \in (0,1)$; cf.~\cite{Lindner2008}. Thus, these GARCH processes always satisfy the above mentioned assumption on $(\alpha(n))$. If $(X_i)$ is even $\beta$- or $\rho$-mixing, then the above mixing condition can be relaxed; cf.\ \cite{ArconesYu1994,ShaoYu1996}.
{\hspace*{\fill}$\Diamond$\par\bigskip}
\end{examplenorm}

\begin{examplenorm}{\em (Strongly dependent data)}\label{example strongly dependent data}
Consider the linear process $X_t:=\sum_{s=0}^\infty a_s\varepsilon_{t-s}$, $t\in\N$, where $(\varepsilon_i)_{i\in\Z}$ are i.i.d.\ random variables on some probability space $(\Omega,{\cal F,\pr})$ with zero mean and finite variance, and the coefficients $a_s$ satisfy $\sum_{s=0}^\infty a_s^2<\infty$. Then $X_1,X_2,\dots$ are identically distributed $L^2(\Omega,{\cal F,\pr})$ random variables. Further, let $\lambda\ge 0$ and assume that the following assertions hold:
\begin{itemize}
    \item[(i)] $a_s=s^{-\beta}\,\ell(s)$, $s\in\N$, where $\beta\in(\frac{1}{2},1)$ and $\ell$ is slowly varying at infinity.
    \item[(ii)] $\ex[|\varepsilon_0|^{2+2\lambda}]<\infty$.
    \item[(iii)] The distribution function $G$ of $\varepsilon_0$ is twice differentiable and satisfies the integrability condition $\sum_{j=1}^2\int |G^{(j)}(x)|^2\phi_{2\lambda}(x)\,dx<\infty$.
\end{itemize}
Under these conditions the covariances $\covi(X_1,X_t)$ are {\em not} summable over $t\in\N$ and thus the process exhibits strong dependence (long-memory). For instance, the infinite moving average representation of an ARFIMA$(p,d,q)$ process with fractional difference parameter $d\in(0,1/2)$ satisfies assumption (a) with $\beta=1-d$; see, for instance, \cite[Section 3]{Hosking1981}. It is shown in \cite[Theorem 2.1]{BeutnerWuZaehle2012} that for the distribution function $F_0$ of the $X_i$, the empirical distribution function $\widehat F_n$ of $X_1,\dots,X_n$ and $\phi_\lambda(x):=(1+|x|)^\lambda$
$$
    n^{\beta-1/2}\,\ell(n)^{-1}\big(\widehat F_n(\cdot)-F_0(\cdot)\big)\,\stackrel{\sf d}{\longrightarrow}\,c_{1,\beta}\, f_0(\cdot) Z\qquad\mbox{(in $(D_{\phi_\lambda,F_0},{\cal D}_{\phi_\lambda,F_0},\|\cdot\|_{\phi_\lambda})$)},
$$
where $f_0$ is the Lebesgue density of $F_0$, $Z$ is a standard normally distributed random variable, and  $c_{1,\beta}:=\{\ex[\varepsilon_0^2](1-(\beta-\frac{1}{2}))(1-(2\beta-1))/(\int_0^\infty(x+x^2)^{-\beta}dx)\}^{1/2}$. Notice that condition (iii) ensures that the distribution function $F_0$ of $X_1$ is differentiable with derivative $f_0\in D_{\phi_\lambda}$; cf.\ inequality (30) in \cite{Wu2003} with $n=\infty$, $\kappa=1$ and $\gamma=2\lambda$. Also notice that the limiting process $c_{1,\beta}f_0(\cdot)Z$ takes values only in the $\|\cdot\|_{\phi}$-separable and ${\cal D}_{\phi,F_0}$-measurable subset of all continuous functions within $D_{\phi,F_0}$.
{\hspace*{\fill}$\Diamond$\par\bigskip}
\end{examplenorm}


\subsection{Strong laws for plug-in estimators}\label{strong laws}

The following theorem generalizes the result of Section 3.2 in \cite{Zaehle2014}.

\begin{theorem}\label{main theorem coroll 2}
Suppose that Assumptions \ref{assumption for kusuoka} and \ref{assumption for qhd} hold, or that Assumption \ref{assumption for qhd - distortion} holds. Let $(\Omega,{\cal F},\pr)$ be a probability space, $\widehat F_n:\Omega\to D$ be a map for every $n\in\N$, and assume that the following conditions hold:
\begin{itemize}
    \item[(a)] $\widehat F_n$ takes values only in $\F_{\cal X}$ and is $({\cal F},{\cal D})$ measurable, $n\in\N$.
    \item[(b)] $\widehat F_n-F_0$ takes values only in $D_{\phi,F_0}$, $n\in\N$.
    \item[(c)] There is some nondecreasing sequence $(r_n)\subset(0,\infty)$ such that
    \begin{equation}\label{condition on empirical difference - 2}
        r_n\|\widehat F_n-F_0\|_\phi\,\longrightarrow\,0\qquad\pr\mbox{-a.s.}
    \end{equation}
\end{itemize}
Then
\begin{equation}\label{empirical difference - conclusion - 2}
    r_n({\cal R}_\rho(\widehat F_n)-{\cal R}_\rho(F_0))\,\longrightarrow\,0\qquad\pr\mbox{-a.s.}
\end{equation}
\end{theorem}

\begin{proof}
Notice that all involved expressions are measurable; cf.\ the proof of Theorem \ref{main theorem coroll}. Assumptions \ref{assumption for kusuoka} and \ref{assumption for qhd} (or Assumption \ref{assumption for qhd - distortion}) and Theorem \ref{theorem qhd of r} ensure that ${\cal R}_\rho$ is quasi-Hadamard differentiable at $F_0$ tangentially to $C_{\phi,F_0}\langle D_{\phi,F_0}\rangle$. By Lemma \ref{HC implies HD}, we may conclude that ${\cal R}_\rho$ is also quasi-Lipschitz continuous at $F_0$ along $D_{\phi,F_0}$ in the sense of Definition \ref{definition quasi continuity}. Thus, (\ref{condition on empirical difference - 2}) obviously implies (\ref{empirical difference - conclusion - 2}).
\end{proof}

The following Examples \ref{generalized GC - independent - mz}--\ref{generalized GC - mixing - mz} illustrate condition (c) of Theorem \ref{main theorem coroll 2}, where $\widehat F_n$ will always be the empirical distribution function $\widehat F_n:=\frac{1}{n}\sum_{i=1}^n\eins_{[X_i,\infty)}$ of the first $n$ variables of a sequence $X_1,X_2,\ldots$ of identically distributed random variables on some probability space $(\Omega,{\cal F},\pr)$.

\begin{examplenorm}{\em (Independent data)}\label{generalized GC - independent - mz}
Assume that $X_1,X_2,\dots$ are i.i.d.\ with distribution function $F_0$. Let $\phi$ be any weight function, and $r\in[0,\frac{1}{2})$. If the sequence $(X_i)$ is i.i.d.\ and $\int\phi^{1/(1-r)}dF<\infty$, then (\ref{condition on empirical difference - 2}) holds for $r_n=n^r$. This is an immediate consequence of Theorem 7.3 in \cite{AndersenGineZinn1988}; cf.\ Theorem 2.1 in \cite{Zaehle2014}.
{\hspace*{\fill}$\Diamond$\par\bigskip}
\end{examplenorm}

\begin{examplenorm}{\em (Weakly dependent data)}\label{generalized GC - mixing - mz}
Suppose that $\int\phi\,dF<\infty$. Further suppose that $(X_i)$ is $\alpha$-mixing with mixing coefficients $\alpha(n)$, let $\alpha(t):=\alpha(\lfloor t\rfloor)$ be the \cadlag\ extension of $\alpha(\cdot)$ from $\N$ to $\R_+$, and assume that $\int_0^1 \log\big(1+\alpha^\rightarrow(s/2)\big)\,\overline{G}\,^\rightarrow(s)\,ds<\infty$
for $\overline G:=1-G$, where $G$ denotes the df of $\phi(X_1)$ and $\overline{G}\,^\rightarrow$ the right-continuous inverse of $\overline{G}$. It was shown in \cite[Theorem 2.3]{Zaehle2014} that, under the imposed assumptions, (\ref{condition on empirical difference - 2}) holds for $r_n=1$. Notice that the integrability condition above holds in particular if $\ex[\phi(X_1)\log^+\phi(X_1)]<\infty$ and $\alpha(n)={\cal O}(n^{-\vartheta})$ for some arbitrarily small $\vartheta>0$; cf.\ \cite[Application 5, p.\,924]{Rio1994}.

Suppose that the sequence $(X_i)$ is $\alpha$-mixing with mixing coefficients $\alpha(n)$. Let $r\in[0,\frac{1}{2})$ and assume that $\alpha(n)\le K n^{-\vartheta}$ for all $n\in\N$ and some constants $K>0$ and $\vartheta>2r$. Then (\ref{condition on empirical difference - 2}) holds for $\phi\equiv 1$ and $r_n=n^r$; cf.\ \cite[Theorem 2.2]{Zaehle2014}.
{\hspace*{\fill}$\Diamond$\par\bigskip}
\end{examplenorm}


\section{Proofs}\label{Proof of QHD of R}


\subsection{Proof of Theorem \ref{theorem kusuoka}}\label{proof of theorem kusuoka}

Assertions (i), (ii) and (iv) are known from  Proposition 5.1 in \cite{BelomestnyKraetschmer2012}; see also independent proof of (ii) and (iv) in \cite{KraetschmerZaehle2011}. Moreover it follows from the last calculation in Section 4.3 of \cite{KraetschmerZaehle2011} that $|g(s)-g(t)|\le g_\rho(|s-t|)$ for every $g\in {\cal G}_\rho$ and $s,t\in [0,1]$. This implies assertion (iii), because all elements of ${\cal G}_{\rho}$ are concave.
Thus it remains to show assertion (v).

For this purpose, let $(g_{k})$ denote any sequence in ${\cal G}_{\rho}$ which converges to some $g\in{\cal G}_{\rho}$ w.r.t.\ the uniform metric, in particular
$g_{k}(F_{0}(x))\to g(F_{0}(x))$ and therefore also $1 - g_{k}(F_{0}(x))\to 1 - g_{F_{0}}(x)$ for every $x\in\R.$ Then by Fatou's lemma
\begin{equation}
\label{negativachse}
\liminf_{k\to\infty}\int_{-\infty}^{0}g_{k}(F_{0}(x))\,dx\geq \int_{-\infty}^{0}g(F_{0}(x))\,dx.
\end{equation}
Furthermore, for any $x > F_{0}^{\leftarrow}(1/2)^{+},$ concavity of $g_{k}$ along with (iii) implies
$$
1 - g_{k}(F_{0}(x))\leq (1- F_{0}(x)) g_{k}'(F_{0}(x))\leq
(1- F_{0}(x)) g_{k}'(1/2)\leq (1- F_{0}(x))\,4\,g_{\rho}(1/4).
$$
This means
\begin{equation}
\label{positivachse}
|(1 - g_{k}(F_{0}))\eins_{[0,\infty)}|\leq
(1 - g_{k}(F_{0}))\eins_{[0,F_{0}^{\leftarrow}(1/2)^{+}]} +
4 g_{\rho}(1/4) (1- F_{0}) \eins_{(F_{0}^{\leftarrow}(1/2)^{+},\infty)}
\end{equation}
Recall that $F_{0}$ is the distribution function of some $\pr$-integrable random variable so that $\int_{0}^{\infty}(1 - F_{0}(x))\,dx < \infty.$ Since in addition $(1 - g_{k}(F_{0}))\eins_{[0,F_{0}^{\leftarrow}(1/2)^{+}]}$ is bounded, it follows that the right-hand side of (\ref{positivachse}) 
is integrable w.r.t.\ the Lebesgue measure on $\R$. Hence in view of (\ref{positivachse}), the application of the Dominated Convergence Theorem yields
\begin{equation}\label{konvergenzpositivachse}
 \lim_{k\to\infty}\int_{0}^{\infty} (1 - g_{k}(F_{0}(x)))\,dx =
 \int_{0}^{\infty}(1 - g(F_{0}(x)))\,dx.
 \end{equation}
Combining (\ref{negativachse}) and (\ref{konvergenzpositivachse}), we may conclude $\liminf_{k\to\infty}{\cal R}_{g_{k}}(F_{0}) \geq {\cal R}_{g}(F_{0})$.

Let us now suppose that in addition $\int_{-\infty}^{0} g_{\rho}(\gamma F_{0}(x))\,dx < \infty$ holds for some $\gamma\in (0,1).$ By concavity of $g_{k}$ along with (ii), we obtain for $x < 0$
$$
g_{k}(x)\leq\frac{g_{k}(\gamma F_{0}(x))}{\gamma}\leq \frac{g_{\rho}(\gamma F_{0}(x))}{\gamma}.
$$
Then we may have $\lim_{k\to\infty}\int_{-\infty}^{0}g_{k}(F_{0}(x))\,dx = \int_{-\infty}^{0}g(F_{0}(x))\,dx$ by the Dominated Convergence Theorem. Due to
(\ref{konvergenzpositivachse}), this implies $
\lim_{k\to\infty}{\cal R}_{g_{k}}(F_{0}) = {\cal R}_{g}(F_{0}),
$
completing the proof.


\subsection{Auxiliary lemma}\label{Auxiliary results}

\begin{lemma}\label{robustrepresentationcontinuous}
Under the Assumptions \ref{assumption for kusuoka} and \ref{assumption for qhd}, the following assertions hold:
\begin{itemize}
    \item[(i)] For every $v\in C_{\phi,F_0}$, the mapping ${\cal G}_{\rho}\rightarrow\R$, $g\mapsto\int_{F^\rightarrow(0)}^{F^\leftarrow(1)} g'(F(x))\,v(x)\,dx$ is continuous w.r.t.\ the uniform metric on ${\cal G}_{\rho}$.
    \item[(ii)] The mapping $D_{\phi,F_0}\to\ell^\infty({\cal G}_\rho)$, $v\mapsto(\int_{F_0^\rightarrow(0)}^{F_0^\leftarrow(1)} g'(F_0(x))\,v(x)\,dx)_{g\in{\cal G}_{\rho}}$ is continuous w.r.t.\ $\|\cdot\|_\phi$, where $\ell^\infty({\cal G}_\rho)$ is the space of all bounded real-valued functions on ${\cal G}_\rho$ equipped with the sup-norm.
\end{itemize}
\end{lemma}

\begin{proof}
We will first prove assertion (i). 
Let $(g_{k})$ be any sequence in ${\cal G}_{\rho}$ which converges to some $g\in{\cal G}_{\rho}$ w.r.t.\ the uniform metric. In view of the Dominated Convergence Theorem, it suffices to show 1) that $g_k'(F_0(x))\,v(x)$ converges to $g_k'(F_0(x))\,v(x)$ for Lebesgue a.e.\ $x\in(F_0^{\rightarrow}(0),F_0^{\leftarrow}(1))$, and 2) that there is some Lebesgue integrable majorant for the sequence $(g_k'(F_0(\cdot))\,v(\cdot))$.

To verify condition 1), let us first fix any $t\in (0,1)$. Then, on the one hand, we may find for every $\varepsilon > 0$ some $\delta\in (0,1-t)$ such that
$$
    g'(t)-\varepsilon\,<\,\frac{g(t + \delta) - g(t)}{\delta}\,=\,\liminf_{k\to\infty}\frac{g_{k}(t + \delta) - g_{k}(t)}{\delta}\,\le\,\liminf_{k\to\infty}g_{k}'(t),
$$
which means that
\begin{equation}\label{eins}
    g'(t)\,\le\,\liminf\limits_{k\to\infty}g_{k}'(t).
\end{equation}
On the other hand, we have for every $\delta\in (0,t)$
$$
    \limsup_{k\to\infty}g_{k}'(t)\,\le\,\limsup_{k\to\infty}\frac{g_{k}(t) - g_{k}(t - \delta)}{\delta}\,=\,\frac{g(t) - g(t - \delta)}{\delta}.
$$
This implies
\begin{equation}\label{zwei}
    \limsup_{k\to\infty}g_{k}'(t)\,\le\,\lim_{\delta\to 0_{+}} \frac{g(t) - g(t - \delta)}{\delta} = g_{-}'(t),
\end{equation}
where $g_{-}'(t)$ denotes the left-sided derivative of $g$ at $t.$ Combining (\ref{eins}) and (\ref{zwei}), we obtain
\begin{equation}\label{drei}
    g'(F_{0}(x))\,=\,\lim_{k\to\infty} g_{k}'(F_{0}(x))\qquad\mbox{for all }x\in \R\setminus A,
\end{equation}
where $A$ denotes the set consisting of all $x\in(F_0^{\rightarrow}(0),F_0^{\leftarrow}(1))$  for which $g_{-}'(F_0(x))$ differs from $g'(F_0(x))$. Notice that $A$ is at most countable due to monotonicity of the left- and the right-sided derivative functions of $g$. Since by assumption $F_0$ is strictly increasing on  $(F^{\rightarrow}(0),F^{\leftarrow}(1))$, this implies $g'(F_0(x))=\lim_k g_k'(F_0(x))$ for all but countably many $x$ in  $(F^{\rightarrow}(0),F^{\leftarrow}(1))$.

To verify condition 2), let $\gamma\in (0,1)$ be as in the integrability condition (b) from Assumption \ref{assumption for qhd}. Then, in view of assertion (iii) from Theorem \ref{theorem kusuoka}, we may conclude
\begin{equation}\label{vier}
    \sup_{k\in\mathbb{N}}|g_{k}'(F_{0}(x)) v(x)|\, \le\, \frac{\sup_{k\in\mathbb{N}}|g_{k}'(F_{0}(x))|\, \|v\|_{\phi}}{\phi(x)}\, \le\,\frac{g_{\rho}(\gamma F_0(x))}{F_0(x)\phi(x)}\,\frac{\|v\|_{\phi}}{\gamma}
\end{equation}
for every $x\in (F_0^{\rightarrow}(0),F_0^{\leftarrow}(1)).$  The integrability condition (b) from Assumption \ref{assumption for qhd} ensures that the right-hand side in (\ref{vier}) is Lebesgue integrable over $(F^{\rightarrow}(0),F^{\leftarrow}(1))$.

Now, we will prove assertion (ii). Let $(v_{k})$ be any sequence in $D_{\phi,F_0}$ which converges to some $v\in D_{\phi,F_0}$ w.r.t.\ $\|\cdot\|_\phi$. We clearly have
\begin{eqnarray*}
    & & \sup_{g\in{\cal G}_\rho}\Big|\int_{F_0^\rightarrow(0)}^{F_0^\leftarrow(1)} g'(F(x))\,v_k(x)\,dx-\int_{F_0^\rightarrow(0)}^{F_0^\leftarrow(1)}g'(F_0(x))\,v(x)\,dx\Big|\\
    & \le & \int_{F_0^\rightarrow(0)}^{F_0^\leftarrow(1)} \sup_{g\in{\cal G}_\rho}\,g'(F(x))|v_k(x)-v(x)|\,dx\,\\
    & \le & \int_{F_0^\rightarrow(0)}^{F_0^\leftarrow(1)} \frac{g_{\rho}(\gamma F(x))}{\gamma F(x)\phi(x)}\,dx\,\|v_k-v\|_\phi,
\end{eqnarray*}
and the integrability condition (b) from Assumption \ref{assumption for qhd} ensures that the latter expression converges to $0$ as $k\to\infty$. The proof is now complete.
\end{proof}


\subsection{Proof of Theorem \ref{theorem qhd of r}}\label{ersterHauptbeweis}

The statistical functional ${\cal R}_\rho:\F_{\cal X}\to\R$ can be represented as composition
$$
    {\cal R}_\rho\,=\,{\cal S}_{\rho}\circ{\cal T}_{\rho}
$$
with ${\cal T}_{\rho}:\F_{\cal X}\to\ell^\infty({\cal G}_\rho)$ and ${\cal S}_{\rho}:\ell^\infty({\cal G}_\rho)\to\R$ given by
$$
    {\cal T}_{\rho}(F)\,:=\,\big({\cal R}_g(F)\big)_{g\in{\cal G}_\rho}
$$
and
$$
    {\cal S}_{\rho}\big((x_g)_{g\in{\cal G}_\rho}\big)\,:=\,\sup_{g\in{\cal G}_\rho}x_g,
$$
respectively, where $\ell^\infty({\cal G}_\rho)$ is the space of all bounded real-valued functions on ${\cal G}_\rho$ equipped with the sup-norm. In the following we will show that the functional ${\cal T}_{\rho}$ is quasi-Hadamard differentiable at $F_0$ tangentially to $D_{\phi,F_0}\langle D_{\phi,F_0}\rangle$ with quasi-Hadamard derivative $\dot {\cal T}_{\rho,F_0}$ given by
\begin{equation}\label{def of qh ableitung von T}
    \dot{\cal T}_{\rho,F_0}(v)\,:=\,\Big(\int_{F_0^\rightarrow(0)}^{F_0^\leftarrow(1)} g'(F_0(x))\,v(x)\,dx\Big)_{g\in{\cal G}_\rho}\,,\qquad v\in C_{\phi,F_0}.
\end{equation}
This is sufficient for the proof of Theorem \ref{theorem qhd of r}. Indeed, it is known from \cite[Proposition 1]{Roemisch2006} that the mapping ${\cal S}_{\rho}$ is Hadamard differentiable at every $(x_g)_{g\in{\cal G}_\rho}$ (tangentially to the whole space $\ell^\infty({\cal G}_\rho)$) with (possibly nolinear) Hadamard derivative $\dot{\cal S}_{\rho,(x_g)_{g\in{\cal G}_\rho}}$ given by
$$
    \dot{\cal S}_{\rho,(x_g)_{g\in{\cal G}_\rho}}\big((w_g)_{g\in{\cal G}_\rho}\big)\,:=\,\lim_{\varepsilon\downarrow 0}\,\sup_{g\in{\cal G}_\rho((x_g)_{g\in{\cal G}_\rho},\varepsilon)} w_g,\qquad (w_g)_{g\in{\cal G}_\rho}\in\ell^\infty({\cal G}_\rho),
$$
where ${\cal G}_{\rho}((x_g)_{g\in{\cal G}_\rho}, \varepsilon)$ denotes the set of all $g\in {\cal G}_{\rho}$ satisfying $\sup_{h\in {\cal G_{\rho}}} x_{h} - \varepsilon\leq x_{g}$. Moreover, the restriction of $F_{0}$ to $(F_{0}^{\rightarrow}(0),F_{0}^{\leftarrow}(1))$ is injective by assumption. Therefore, in view of (i) in Theorem \ref{theorem kusuoka}, we obtain for $v\in C_{\phi,F_{0}}$
\begin{eqnarray*}
    & & \dot{\cal S}_{\rho,({\cal R}_{g}(F_{0}))_{g\in{\cal G}_\rho}}\Big(\Big(\int_{F_{0}^{\rightarrow}(0)}^{F_{0}^{\leftarrow}(1)}g'(F_{0}(x)) v(x)\, dx\Big)_{g\in{\cal G}_\rho}\Big)
     = \lim_{\varepsilon\downarrow 0}\,\sup_{g\in{\cal G}_\rho(F_0,\varepsilon)}\int_{F_{0}^{\rightarrow}(0)}^{F_{0}^{\leftarrow}(1)}g'(F_{0}(x)) v(x)\, dx
\end{eqnarray*}
If in addition $\int_{-\infty}^{0}g_{\rho}(\delta F_0(x))\,dx < \infty$ holds for some $\delta\in (0,1)$, then by parts (i) and (v) of Theorem \ref{theorem kusuoka}, the set ${\cal G}_\rho(F_{0})$ is nonempty, and
$$
    \dot{\cal S}_{\rho,({\cal R}_{g}(F_{0}))_{g\in{\cal G}_\rho}}\Big(\Big(\int_{F_{0}^{\rightarrow}(0)}^{F_{0}^{\leftarrow}(1)}g'(F_{0}(x)) v(x)\, dx\Big)_{g\in{\cal G}_\rho}\Big)\,=\,\sup_{g\in{\cal G}_\rho(F_{0})} \int_{F_{0}^{\rightarrow}(0)}^{F_{0}^{\leftarrow}(1)}g'(F_{0}(x)) v(x)\, dx
$$
for every $v\in C_{\phi,F_{0}}.$ Hence the full claim of Theorem \ref{theorem qhd of r} follows from the chain rule in Lemma \ref{lemma chain rule}.

We are now going to establish the above mentioned quasi-Hadamard differentiability of ${\cal T}_{\rho}$. First of all notice that the map $\dot{\cal T}_{\rho,F_0}$ defined in (\ref{def of qh ableitung von T}) is continuous w.r.t.\ $\|\cdot\|_\phi$ by part (ii) of Lemma \ref{robustrepresentationcontinuous}. Now, let $(v,(v_n),(h_n))$ be a triplet with $v\in C_{\phi,F_0}$, $(v_n)\subset D_{\phi,F_0}$ satisfying $\|v_n-v\|_{\phi}\to 0$ and $(F_0+h_nv_n)\subset\F_{\cal X}$, and $(h_n)\subset(0,\infty)$ satisfying $h_n\to 0$. We have to show that
$$
    \lim_{n\to\infty}\,\Big\|\frac{{\cal T}_\rho(F_0+h_nv_n)-{\cal T}_\rho(F_0)}{h_n}-\dot{\cal T}_{\rho,F_0}(v)\Big\|_\infty\,=\,0,
$$
that is,
\begin{equation}\label{proof qhd of t - 10}
    \lim_{n\to\infty}\,\sup_{g\in{\cal G}_\rho}\,\Big|\frac{{\cal R}_g(F_0+h_nv_n)-{\cal R}_g(F_0)}{h_n}-\int_{F_0^\rightarrow(0)}^{F_0^\leftarrow(1)} g'(F_0(x))\,v(x)\,dx\Big|\,=\,0
\end{equation}
or, in other words,
\begin{equation}\label{proof qhd of t - 15}
    \lim_{n\to\infty}\,\sup_{g\in{\cal G}_\rho}\,\Big|\int_{F_0^\rightarrow(0)}^{F_0^\leftarrow(1)}\frac{g\big((F_0+h_nv_n)(x)\big)-g\big(F_0(x)\big)}{h_n}\,dx-\int_{F_0^\rightarrow(0)}^{F_0^\leftarrow(1)} g'(F_0(x))\,v(x)\,dx\Big|\,=\,0.
\end{equation}
By Assumption \ref{assumption for qhd}\,(b), there exist $r\in\N_0$ and  $F_0^{\rightarrow}(0)=:a_0<a_1<\cdots<a_{r+1}:=F_0^{\leftarrow}(1)$ such that for every $i=1,\ldots,r$ the restriction $F_0|_{(a_{i-1},a_i)}$ of $F_0$ to $(a_{i-1},a_i)$ is continuously differentiable with strictly positive derivative. For every $i=0,\ldots,r$, we consider the extension $F_{0,i}:=F_{0}\eins_{[a_{i},a_{i+1})}+\eins_{[a_{i+1},\infty)}$ of $F_0|_{(a_{i},a_{i+1})}$ from $(a_{i},a_{i+1})$ to $\R$. This extension $F_{0,i}$ is contained in $\F_{\cal X}$. Indeed, if $X_0$ is any random variable from ${\cal X}$ with distribution function $F_{0}$, then $X_{0,i}:=a_{i}\vee (X_0\wedge a_{i+1})$ belongs to ${\cal X}$ (recall that ${\cal X}$ is a Stonean vector lattice) and $F_{0,i}$ is the distribution function of $X_{0,i}$. Further, for every $g\in{\cal G}_\rho$ we have that
\begin{eqnarray*}
    \lefteqn{\Big|\int_{F_0^\rightarrow(0)}^{F_0^\leftarrow(1)}\frac{g\big((F_0+h_nv_n)(x)\big)-g\big(F_0(x)\big)}{h_n}\,dx-\int_{F_0^\rightarrow(0)}^{F_0^\leftarrow(1)} g'(F_0(x))\,v(x)\,dx\Big|}\\
    & \le & \sum_{i=0}^{r}\Big|\int_{F_{0,i}^{\rightarrow}(0)}^{F_{0,i}^{\leftarrow}(1)}\frac{g\big((F_{0,i}+h_n v_{n,i})(x)\big)-g\big(F_{0,i}(x)\big)}{h_n}\,dx-\int_{F_{0,i}^\rightarrow(0)}^{F_{0,i}^\leftarrow(1)}g'(F_{0,i}(x))\,v_{i}(x)\,dx\Big|,
\end{eqnarray*}
with $v_{n,i}:=v_n\eins_{[a_i,a_{i+1})}$ and $v_{i}:=v\eins_{[a_i,a_{i+1})}$. Thus, for (\ref{proof qhd of t - 15}) it suffices to show that
$$
    \lim_{n\to\infty}\,\sup_{g\in{\cal G}_\rho}\,\Big|\int_{F_{0,i}^\rightarrow(0)}^{F_{0,i}^\leftarrow(1)}\frac{g\big((F_{0,i}+h_nv_{n,i})(x)\big)-g\big(F_{0,i}(x)\big)}{h_n}\,dx
    - \int_{F_{0,i}^\rightarrow(0)}^{F_{0,i}^\leftarrow(1)} g'(F_{0,i}(x))\,v(x)\,dx\Big|=0
$$
for every $i=1,\ldots,r$. So, since $v_{n,i}\in D_{\phi,F_{0,i}}$, $v_{i}\in C_{\phi,F_{0,i}}$, and the restriction of $F_{0,i}$ to $(\F_{0,i}^{\rightarrow}(0),F_{0,i}^{\leftarrow}(1))$ is continuously differentiable with strictly positive derivative, we may without loss of generality restrict ourselves to the case $r=0$. In the remainder of the proof we will show (\ref{proof qhd of t - 15}) for $r=0$.

Let $(t_k)$ be any sequence in $(0,1/2)$ with $t_k\downarrow 0$. Moreover, let the map $g_{k}:[0,1]\to[0,1]$ be defined by
\begin{equation}\label{pigk}
    g_{k}(t)\,:=\,\int_0^{t}\eins_{[t_{k},1-t_{k}]}(s)\,g'(s)\,ds,\qquad t\in[0,1],
\end{equation}
and notice that $g(t)=\lim_{k\to\infty}g_k(t)$ for every $t\in[0,1]$. Then, if we set
\begin{eqnarray}
    a_{n}(g)   & := & \int_{F_0^\rightarrow(0)}^{F_0^\leftarrow(1)}\frac{g\big((F_0+h_nv_n)(x)\big)-g\big(F_0(x)\big)}{h_n}\,dx,\label{ang}\\
    a_{n,k}(g) & := & \int_{F_0^\rightarrow(0)}^{F_0^\leftarrow(1)}\frac{g_{k}\big((F_0+h_nv_n)(x)\big)-g_{k}\big(F_0(x)\big)}{h_n}\,dx\label{ankg}
\end{eqnarray}
and
\begin{eqnarray}
    b(g)   & := & \int_{F_0^\rightarrow(0)}^{F_0^\leftarrow(1)} g'(F_0(x))\,v(x)\,dx,\label{bg}\\
    b_k(g) & := & \int_{F_0^\rightarrow(0)}^{F_0^\leftarrow(1)} g'(F_0(x))\,v(x)\,\eins_{[t_k,1-t_k]}(F_0(x))\,dx.\label{bkg}
\end{eqnarray}
Of course, (\ref{proof qhd of t - 15}) follow if we can show that
\begin{eqnarray}
    \lim_{k\to\infty}\,\limsup_{n\to\infty}\,\sup_{g\in{\cal G}_\rho}\,|a_n(g)-a_{n,k}(g)| & = & 0,\label{proof qhd of t - 20}\\
    \lim_{n\to\infty}\,\sup_{g\in{\cal G}_\rho}\,|a_{n,k}(g)-b_k(g)| & = & 0\qquad\mbox{for all }k\in\N,\label{proof qhd of t - 30}\\
    \lim_{k\to\infty}\,\sup_{g\in{\cal G}_\rho}\,|b_k(g)-b(g)| & = & 0.\label{proof qhd of t - 40}
\end{eqnarray}
We will now verify in Steps 1--3 that (\ref{proof qhd of t - 20})--(\ref{proof qhd of t - 40}) hold true.

{\em Step 1}. We first show (\ref{proof qhd of t - 40}). We clearly have
$$
    \sup_{g\in{\cal G}_\rho}\,|b_k(g)-b(g)|\,\le\,\int_{F_0^\rightarrow(0)}^{F_0^\leftarrow(1)}\sup_{g\in{\cal G}_\rho}\,g'(F_0(x))\,|v(x)|\,\eins_{(0,t_k)\cup(1-t_k,1)}(F_0(x))\,dx,
$$
and the latter integrand converges to zero as $k\to\infty$ for every $x\in(F_0^\rightarrow(0),F_0^\leftarrow(1))$. By part (iii) of Theorem \ref{theorem kusuoka} we also have
$$
    \sup_{g\in{\cal G}_\rho}\,g'(F_0(x))\,|v(x)|\,\eins_{(0,t_k)\cup(1-t_k,1)}(F_0(x))\,\le\,\frac{g_\rho(\widetilde\gamma F_0(x))}{\widetilde\gamma F_0(x)}\,|v(x)|\,\le\,
    \frac{g_\rho(\widetilde\gamma F_0(x))}{\widetilde\gamma F_0(x)\phi(x)}\,\|v\|_\phi
$$
for every $x\in(F_0^\rightarrow(0),F_0^\leftarrow(1))$ and all $\widetilde\gamma\in(0,1)$, and so the integrability condition (b) in Assumption \ref{assumption for qhd} ensures that we may apply the Dominated Convergence Theorem to obtain (\ref{proof qhd of t - 40}).

{\em Step 2}. We next show (\ref{proof qhd of t - 30}). According to Lemma A.1 in \cite{BelomestnyKraetschmer2012} we have
\begin{eqnarray*}
   a_{n,k}(g)
   & = & \int_{F_0^\rightarrow(0)}^{F_0^\leftarrow(1)}\frac{g_k\big((F_0+h_nv_n)(x)\big)-g_k\big(F_0(x)\big)}{h_n}\,dx\\
   & = & \int_0^1\frac{(F_0+h_nv_n)^\leftarrow(t)-F_0^\leftarrow(t)}{h_n}\,dg_k(t)\\
   & = & \int_0^1\frac{(F_0+h_nv_n)^\leftarrow(s)-F_0^\leftarrow(s)}{h_n}\,\eins_{[t_k,1-t_k]}(t)\,g'(t)\,dt.
\end{eqnarray*}
By a change-of-variable $t:=F_0(x)$ and Assumption \ref{assumption for qhd}\,(a), we also have
\begin{eqnarray*}
   b_k(g)
   & = & \int_{F_0^\rightarrow(0)}^{F_0^\leftarrow(1)} g'(F_0(x))\,v(x)\,\eins_{[t_k,1-t_k]}(F_0(x))\,dx\\
   & = & \int_0^1 g'(t)\,v(F_0^\leftarrow(t))\,\eins_{[t_k,1-t_k]}(t)\,\frac{1}{F_0'(F_0^\leftarrow(t))}\,dt.
\end{eqnarray*}
Thus, using part (iii) of Theorem \ref{theorem kusuoka}, we have
\begin{eqnarray*}
   \lefteqn{\sup_{g\in{\cal G}_\rho}\,|a_{n,k}(g)-b_k(g)|}\\
   & \le & \sup_{g\in{\cal G}_\rho}\, \int_0^1\Big|\frac{(F_0+h_nv_n)^\leftarrow(t)-F_0^\leftarrow(t)}{h_n}-\frac{v(F_0^\leftarrow(t))}{F_0'(F_0^\leftarrow(t))}\Big|\,\eins_{[t_k,1-t_k]}(t)\,g'(t)\,dt\\
   & \le & \sup_{t\in[t_k,1-t_k]}\Big|\frac{(F_0+h_nv_n)^\leftarrow(t)-F_0^\leftarrow(t)}{h_n}-\frac{v(F_0^\leftarrow(t))}{F_0'(F_0^\leftarrow(t))}\Big|\,\int_{t_k}^{1-t_k}\sup_{g\in{\cal G}_\rho}\,g'(t)\,dt\\
   & \le & C_k \sup_{t\in[t_k,1-t_k]}\Big|\frac{(F_0+h_nv_n)^\leftarrow(t)-F_0^\leftarrow(t)}{h_n}-\frac{v(F_0^\leftarrow(t))}{F_0'(F_0^\leftarrow(t))}\Big|
\end{eqnarray*}
for $C_k:=\int_{t_k}^{1-t_k}\frac{g_\rho(\widetilde\gamma t)}{\widetilde\gamma t}\,dt<\infty$. Now, part (i) of Lemma 21.4 in \cite{vanderVaart1998} yields that the latter expression converges to zero as $n\to\infty$ for every fixed $k\in\N$.

{\em Step 3}. Finally we will show (\ref{proof qhd of t - 20}). Let $I_n(x)$ denote the closed interval with boundary points $F_0(x)$ and $(F_0+h_nv_n)(x)$. Then we have
\begin{eqnarray}
   \lefteqn{\limsup_{n\to\infty}\,\sup_{g\in{\cal G}_\rho}\,|a_n(g)-a_{n,k}(g)|}\nonumber\\
   & \le & \limsup_{n\to\infty}\,\int \eins_{(F_0^\rightarrow(0),F_0^\leftarrow(1))}(x)\sup_{g\in{\cal G}_\rho}\,\int_{I_n(x)} \frac{\eins_{[0,t_k]\cup[1-t_k,1]}(s)\,g'(s)}{h_n}\,ds\,dx.
   \label{proof qhd of t - 24}
\end{eqnarray}
Notice that the integrand of the $dx$-integral, i.e.
\begin{equation}\label{proof qhd of t - 25}
    G_{n,k}(x)\,:=\,\eins_{(F_0^\rightarrow(0),F_0^\leftarrow(1))}(x)\,\sup_{g\in{\cal G}_\rho}\,\int_{I_n(x)}\frac{\eins_{[0,t_k]\cup[1-t_k,1]}(s)\,g'(s)}{h_n}\,ds,\qquad x\in\R,
\end{equation}
is clearly nonnegative, and measurable w.r.t.\ the Borel $\sigma$-algebra ${\cal B}(\R)$ because its restriction to $\R\setminus\D_n$ (with $\D_n$ the set of discontinuity points of $v_n$) is lower semi-continuous w.r.t.\ the relative topology. In Step 4 below we will show that the integrand $G_{n,k}$ is bounded above by the nonnegative and ${\cal B}(\R)$-measurable function
\begin{equation}\label{proof qhd of t - 28}
    G(x)\,:=\,\Big(\frac{\|v\|_\phi+c}{\gamma}\Big)\,\frac{g_\rho\big(\gamma F_0(x)\big)}{F_0(x)\,\phi(x)},\qquad x\in\R,
\end{equation}
where $c\in(0,\infty)$ is some suitable constant being independent of $n$ (and $k$), and $\gamma$ is as in condition (b) of Assumption \ref{assumption for qhd}. By condition (b) of Assumption \ref{assumption for qhd}, the mapping $G$ is even integrable w.r.t.\ the Lebesgue measure on $\R$. So, applying Fatou's lemma to the sequence $(G-G_{n,k})_{n},$ we obtain from (\ref{proof qhd of t - 24})
\begin{eqnarray}
   \lefteqn{\limsup_{n\to\infty}\,\sup_{g\in{\cal G}_\rho}\,|a_n-a_{n,k}(g)|}\nonumber\\
   & \le & \int \eins_{(F_0^\rightarrow(0),F_0^\leftarrow(1))}(x)\,\limsup_{n\to\infty}\,\sup_{g\in{\cal G}_\rho}\,\int_{I_n(x)} \frac{\eins_{[0,t_k]\cup[1-t_k,1]}(s)\,g'(s)}{h_n}\,ds\,dx.
   \label{proof qhd of t - 26}
\end{eqnarray}
Step 4 below also shows that $G$ defined in (\ref{proof qhd of t - 28}) provides a ${\cal B}(\R)$-measurable majorant of the integrand of the latter $dx$-integral, i.e.\ of
$$
    G_{k}(x)\,:=\,\eins_{(F_0^\rightarrow(0),F_0^\leftarrow(1))}(x)\,\limsup_{n\to\infty}\,\sup_{g\in{\cal G}_\rho}\,\int_{I_n(x)}\frac{\eins_{[0,t_k]\cup[1-t_k,1]}(s)\,g'(s)}{h_n}\,ds,\qquad x\in\R,
$$
and by the integrability condition (b) in Assumption \ref{assumption for qhd} the majorant $G$ is also $dx$-integrable. So, in view of the Dominated Convergence Theorem, it remains to show that $G_{k}(x)$ converges to zero as $k\to\infty$ for $dx$-almost all $x\in\R$. To do so, let $x\in(F_0^\rightarrow(0),F_0^\leftarrow(1))$. Then, by part (iii) of Theorem \ref{theorem kusuoka} and a change-of-variable $y:=F_0^{-1}(s)$ along with Assumption \ref{assumption for qhd}\,(a), we obtain
\begin{eqnarray*}
    \lefteqn{\limsup_{n\to\infty}\,\sup_{g\in{\cal G}_\rho}\,\int_{I_n(x)}\frac{\eins_{[0,t_k]\cup[1-t_k,1]}(s)\,g'(s)}{h_n}\,ds}\\
    & \le & \limsup_{n\to\infty}\,\int_0^1\eins_{I_n(x)}(s)\,\eins_{[0,t_k]\cup[1-t_k,1]}(s)\,\frac{g_\rho(\gamma s)}{\gamma s\,h_n}\,ds\\
    & = & \limsup_{n\to\infty}\,\int_{F_0^\rightarrow(0)}^{F_0^\leftarrow(1)}\eins_{I_n(x)}(F_0(y))\,\eins_{[0,t_k]\cup[1-t_k,1]}(F_0(y))\,\frac{g_\rho(\gamma F_0(y))}{\gamma F_0(y)\,h_n}\,F_0'(y)\,dy.
\end{eqnarray*}
Now, if $k$ is sufficiently large so that $F_0(x)\in(t_k,1-t_k)$, then also $I_n(x)\subset(t_k,1-t_k)$ for $n$ sufficiently large. That is, for sufficiently large $k$ we have that the latter expression equals zero. This implies $G_k(x)\to 0$ as $k\to\infty$ for all $x\in\R$.

{\em Step 4}. Let $G_{n,k}$ and $G$ be defined as in (\ref{proof qhd of t - 25}) and (\ref{proof qhd of t - 28}), respectively. It remains to show that $G_{n,k}\le G$. By the concavity of all $g\in{\cal G}_\rho$ and part (iii) of Theorem \ref{theorem kusuoka}  we obtain for every $x\in(F_0^\rightarrow(0),F_0^\leftarrow(1))$ with $(F_0+h_nv_n)(x)>0$ and $v_n(x)\ge 0$
\begin{eqnarray}
   \sup_{g\in{\cal G}_\rho}\,\int_{I_n(x)}\frac{\eins_{[0,t_k]\cup[1-t_k,1]}(s)\,g'(s)}{h_n}\,ds
   & \le & \frac{\sup_{g\in{\cal G}_\rho}g'\big(F_0(x)\big)}{h_n}\,h_nv_n(x)\nonumber\\
   & \le & \frac{g_\rho\big(\gamma F_0(x)\big)\,v_n(x)\,\phi(x)}{\gamma F_0(x)\,\phi(x)}\nonumber\\
   & \le & \|v_n\|_\phi\,\frac{g_\rho\big(\gamma F_0(x)\big)}{\gamma F_0(x)\,\phi(x)}\nonumber\\
   & \le & \Big(\frac{\|v\|_\phi+c}{\gamma}\Big)\,\frac{g_\rho\big(\gamma F_0(x)\big)}{F_0(x)\,\phi(x)}\label{proof qhd of t - 50}
\end{eqnarray}
for some suitable constant $c\in(0,\infty)$ being independent of $n\in\N$, and $\gamma$ as in condition (b) of Assumption \ref{assumption for qhd}. For every $x\in(F_0^\rightarrow(0),F_0^\leftarrow(1))$ with $(F_0+h_nv_n)(x)>0$ and $v_n(x)<0$ we further obtain by the concavity of all $g\in{\cal G}_\rho$ and part (ii) of Theorem \ref{theorem kusuoka}
\begin{eqnarray}
   \sup_{g\in{\cal G}_\rho}\,\int_{I_n(x)}\frac{\eins_{[0,t_k]\cup[1-t_k,1]}(s)\,g'(s)}{h_n}\,ds
   & \le & \sup_{g\in{\cal G}_\rho}\,\int_{(F_0+h_nv_n)(x)}^{F_0(x)}\frac{\,g'(s)}{h_n}\,ds\nonumber\\
   & =   & \sup_{g\in{\cal G}_\rho}\,\frac{g(F_0(x))-g((F_0+h_nv_n)(x))}{h_n}\nonumber\\
   & \le & \sup_{g\in{\cal G}_\rho}\,\frac{g(F_0(x))-\frac{(F_0+h_nv_n)(x)}{F_0(x)}\,g(F_0(x))}{h_n}\nonumber\\
   & \le & \sup_{g\in{\cal G}_\rho}\,\frac{|v_n(x)|\,g(F_0(x))}{F_0(x)}\nonumber\\
   & \le & \sup_{g\in{\cal G}_\rho}\,\frac{|v_n(x)|\,g(\gamma F_0(x))}{\gamma F_0(x)\,\phi(x)}\nonumber\\
   & \le & \sup_{g\in{\cal G}_\rho}\,\frac{|v_n(x)\phi(x)|\,g(\gamma F_0(x))}{\gamma F_0(x)\,\phi(x)}\nonumber\\
   & \le & \|v_n\|_\phi\sup_{g\in{\cal G}_\rho}\,\frac{g(\gamma F_0(x))}{\gamma F_0(x)\,\phi(x)}\nonumber\\
   & \le & \Big(\frac{\|v\|_\phi+c}{\gamma}\Big)\,\frac{g_\rho(\gamma F_0(x))}{F_0(x)\,\phi(x)}\label{proof qhd of t - 60}
\end{eqnarray}
for some suitable constant $c\in(0,\infty)$ being independent of $n\in\N$, and $\gamma$ as in condition condition (b) of Assumption \ref{assumption for qhd}. Finally, for every $x\in(F_0^\rightarrow(0),F_0^\leftarrow(1))$ with $(F_0+h_nv_n)(x)=0$ (i.e.\ in particular with $h_n=F_0(x)/|v_n(x)|$) we obtain by the concavity of all $g\in{\cal G}_\rho$ and part (ii) of Theorem \ref{theorem kusuoka}
\begin{eqnarray}
   \sup_{g\in{\cal G}_\rho}\,\int_{I_n(x)}\frac{\eins_{[0,t_k]\cup[1-t_k,1]}(s)\,g'(s)}{h_n}\,ds
   & \le & \sup_{g\in{\cal G}_\rho}\,\int_{0}^{F_0(x)}\frac{\,g'(s)}{h_n}\,ds\nonumber\\
   & =   & \sup_{g\in{\cal G}_\rho}\,\frac{g(F_0(x))}{h_n}\nonumber\\
   & =   & \sup_{g\in{\cal G}_\rho}\,\frac{g(F_0(x))\,|v_n(x)|}{F_0(x)}\nonumber\\
   & \le & \sup_{g\in{\cal G}_\rho}\,\frac{g(\gamma F_0(x))\,|v_n(x)|}{\gamma F_0(x)}\nonumber\\
   & =   & \sup_{g\in{\cal G}_\rho}\,\frac{g(\gamma F_0(x))\,|v_n(x)\phi(x)|}{\gamma F_0(x)\,\phi(x)}\nonumber\\
   & \le & \|v_n\|_\phi\,\sup_{g\in{\cal G}_\rho}\,\frac{g(\gamma F_0(x))}{\gamma F_0(x)\,\phi(x)}\nonumber\\
   & \le & \Big(\frac{\|v\|_\phi+c}{\gamma}\Big)\,\frac{g_\rho(\gamma F_0(x))}{F_0(x)\,\phi(x)}\label{proof qhd of t - 70}
\end{eqnarray}
for some suitable constant $c\in(0,\infty)$ being independent of $n\in\N$, and $\gamma$ as in condition condition (b) of Assumption \ref{assumption for qhd}. So we indeed have $G_{n,k}\le G$. This completes the proof of Theorem \ref{theorem qhd of r}.


\subsection{Proof of Theorem \ref{special case drm}}\label{zweiterHauptbeweis}

Let $(v_{n})$ be a sequence in $D_{\phi,F_{0}}$ with $\|v_{n} - v\|_{\phi}\to 0$ for some $v\in D_{\phi,F_{0}},$ and let $(h_{n})$ be a sequence in $(0,\infty)$ with $h_{n}\to 0$ such that $F_{0}+h_{n}v_{n}\in\mathbb{F}_{{\cal X}}$ for every $n$. We have to show
\begin{equation}\label{ableitung distortion risk measure}
    \lim_{n\to\infty}\,\Big|\int_{F_0^\rightarrow(0)}^{F_0^\leftarrow(1)}\frac{g\big((F_0+h_nv_n)(x)\big)-g\big(F_0(x)\big)}{h_n}\,dx-\int_{F_0^\rightarrow(0)}^{F_0^\leftarrow(1)} g'(F_0(x))\,v(x)\,dx\Big|\,=\,0.
\end{equation}
For $\gamma\in (0,1)$ as in condition (b) of Assumption \ref{assumption for qhd - distortion} let us fix a  sequence $(t_{k})$ in $(0,\gamma\wedge \frac12)$ with $t_{k}\downarrow 0$ and $F^{\leftarrow}(t_{k}) = F^{\rightarrow}(t_{k})$ as well as $F^{\leftarrow}(1-t_{k}) = F^{\rightarrow}(1-t_{k})$.  Using the
notations $g_{k}$, $a_{n}(g)$, $a_{n,k}(g)$, $b(g)$ and $b_{k}(g)$ as defined respectively by (\ref{pigk}), (\ref{ang}), (\ref{ankg}), (\ref{bg}), (\ref{bkg}), we may adopt the line of reasoning from the proof of Theorem \ref{theorem qhd of r}. That is, is suffices to show that the following analogues of (\ref{proof qhd of t - 20})--(\ref{proof qhd of t - 40}) hold:
\begin{eqnarray}
    \lim_{k\to\infty}\,\limsup_{n\to\infty}\,|a_n(g)-a_{n,k}(g)| & = & 0,\label{proof qhd of t - 20 - dist risk meas}\\
    \lim_{n\to\infty}\,|a_{n,k}(g)-b_k(g)| & = & 0\qquad\mbox{for all }k\in\N,\label{proof qhd of t - 30 - dist risk meas}\\
    \lim_{k\to\infty}\,|b_k(g)-b(g)| & = & 0.\label{proof qhd of t - 40 - dist risk meas}
\end{eqnarray}
Assertion (\ref{proof qhd of t - 20 - dist risk meas}) can be shown exactly in the same way as (\ref{proof qhd of t - 20}) was shown in Steps 3--4 in Section \ref{ersterHauptbeweis}, where in the present setting ${\cal G}_\rho$ reduces to the singleton $\{g\}$. By continuity and concavity of $g$, we have
\begin{equation}\label{AbschaetzungAbleitung}
    g'(t)\,\le\,\frac{\int_{(1- \lambda) t}^{t} g'(s)\,ds}{\lambda t}\,=\,\frac{g(t) - g((1- \lambda) t)}{\lambda t}\,\le\,\frac{g(\lambda t)}{\lambda t}\qquad\mbox{for all }t,\lambda\in (0,1).
\end{equation}
Then (\ref{proof qhd of t - 40 - dist risk meas}) can be derived in the same way as (\ref{proof qhd of t - 40}) was derived in Step 1 in Section \ref{ersterHauptbeweis}.

Thus, for (\ref{ableitung distortion risk measure}) it remains to show (\ref{proof qhd of t - 30 - dist risk meas}). For fixed $k$, and arbitrary $n$, we have
\begin{eqnarray}
    \lefteqn{|a_{n,k}(g)-b_k(g)|}\label{beweis von fehlenden aussage - 10}\\
    & \le & \int_{F_0^\rightarrow(0)}^{F_0^\leftarrow(1)}\Big|\frac{g_k\big((F_0+h_nv_n)(x)\big)-g_k\big(F_0(x)\big)}{h_n}-g'(F_0(x))\,v(x)\,\eins_{[t_k,1-t_k]}(F_0(x))\Big|\,dx.\nonumber
\end{eqnarray}
We intend to apply the Dominated Convergence Theorem to conclude that the latter integral converges to $0$ as $n\to\infty$. We first show that the integrand converges to zero (as $n\to\infty$) Lebesgue a.e. To this end, let $x\in (F_0^\rightarrow(0),F_0^\leftarrow(1))$. If $g$ is differentiable at $F_{0}(x)$, and $F_{0}(x)$ belongs to $(t_{k},1 - t_{k+1})$, then $(F_0+h_nv_n)(x)\in (t_{k},1 - t_{k+1})$ for sufficient large $n$. This implies that the expression
\begin{eqnarray*}
    \lefteqn{\frac{g_k\big((F_0+h_nv_n)(x)\big)-g_k\big(F_0(x)\big)}{h_n}}\\
    & = & \frac{\big[g\big((F_0+h_nv_n)(x)\big)- g_k(t_{k})\big] - \big[g(\big(F_0(x)\big) - g_k(t_{k})\big]}{h_{n}}\\
    & = & \frac{g\big((F_0+h_nv_n)(x)\big) - g(\big(F_0(x)\big)}{h_{n}}
\end{eqnarray*}
converges to $g'(F_{0}(x)) v(x)$ as $n\to\infty$. If $F_{0}(x)$ does not belong to $[t_{k},1 - t_{k+1}]$, then $(F_0+h_nv_n)(x)\in\R\setminus [t_{k},1 - t_{k+1}]$ for sufficient large $n$.  That means that
$$
    \frac{g_k\big((F_0+h_nv_n)(x)\big)-g_k\big(F_0(x)\big)}{h_n}\,=\,0\quad\mbox{for sufficient large}~n.
$$
Finally, the set $F_{0}^{-1}(\{t_{k},1 - t_{k}\})$ is finite by assumption on $t_{k}.$  To summarize, in view of Assumption \ref{assumption for qhd - distortion}\,(a), we may find some Borel subset $A\subset\R$ of Lebesgue measure zero such that
\begin{eqnarray*}
    \lim_{n\to\infty}\frac{g_k\big((F_0+h_nv_n)(x)\big)-g_k\big(F_0(x)\big)}{h_n}\,=\,\eins_{[t_{k},1-t_{k}]}(F_{0}(x))\,g'(F_{0}(x))\,v(x)
\end{eqnarray*}
holds for $x\in (F_0^\rightarrow(0),F_0^\leftarrow(1))\setminus A$.

In the remainder of the proof, we will show that
\begin{equation}\label{def der majoranten}
    G_k(x)\,:=\,c_k\,\frac{g(\gamma F_{0}(x))}{F_{0}(x)\phi(x)}\,,\qquad x\in\R
\end{equation}
provides an integrable majorant for the integrand of the integral on the right-hand side in (\ref{beweis von fehlenden aussage - 10}), where $c_k\in(0,\infty)$ is some suitable constant being independent of $n$, and $\gamma$ is as in condition (b) of Assumption \ref{assumption for qhd - distortion}. Then (\ref{proof qhd of t - 30 - dist risk meas}) will follow from an application of the Dominated Convergence Theorem. The integrability of $G$ follows from Assumption \ref{assumption for qhd - distortion}\,(b). To show that $G$ is dominating we first use the concavity of $g$ to get for $x\in (F_0^\rightarrow(0),F_0^\leftarrow(1))$ that
\begin{eqnarray*}
    \Big|\frac{g_k\big((F_0+h_nv_n)(x)\big)-g_k\big(F_0(x)\big)}{h_n}\Big|
    & = & \frac{1}{h_n}\,\Big|\int_{F_{0}(x)}^{F_{0}(x) + h_{n} v_{n}(x)}\eins_{[t_{k},1-t_{k}]}(t)\, g'(t)\,dt\Big|\\
    & \le & \frac{1}{h_n}\,\Big|g'(t_{k})\int_{F_{0}(x)}^{F_{0}(x) + h_{n} v_{n}(x)}\eins_{[t_{k},1-t_{k}]}(t)\,dt\Big|\\
    & \le & g'(t_{k})\,|v_{n}(x)|.
\end{eqnarray*}
By assumption, we have $t_{k}\le\gamma$. So we may conclude from (\ref{AbschaetzungAbleitung}) and the monotonicity of $g$ that for $x\in(F_0^\rightarrow(0),F_0^\leftarrow(1))$
$$
    g'(t_{k})\,|v_{n}(x)|\,\le\,\frac{g(F_{0}(x) t_{k})\,|v_{n}(x)|}{F_{0}(x) t_{k}}\,\le\,\frac{g(F_{0}(x) \gamma)\,\|v_{n}\|_{\phi}}{F_{0}(x) t_{k}\,\phi(x)}.
$$
Analogously we obtain for $x\in(F_0^\rightarrow(0),F_0^\leftarrow(1))$
$$
    g'(F_0(x))\,v(x)\,\eins_{[t_k,1-t_k]}(F_0(x))\,\le\,\frac{g(F_{0}(x) \gamma)\,\|v\|_{\phi}}{F_{0}(x) \gamma\,\phi(x)}.
$$
That is, for $c_k:=\sup_{n}\|v_{n}\|_{\phi}/t_{k}+\|v\|/\gamma$ the function $G_k$ defined in (\ref{def der majoranten}) indeed provides a Lebesgue integrable majorant for the integrand of the integral on the right-hand side in (\ref{beweis von fehlenden aussage - 10}); notice that the sequence $(\|v_n\|_\phi)$ is bounded by assumption. This completes the proof of Theorem \ref{special case drm}.


\appendix


\section{Quasi-Hadamard differentiability and quasi-Lipschitz continuity}\label{appendix QHD and FDM}

The following definition recalls from \cite{BeutnerZaehle2010} the notion of quasi-Hadamard differentiability. Let $\V$, $\V'$ and $\V''$ be vector spaces, and $\V_0$ be a subspace of $\V$. Let $\|\cdot\|_{\V_0}$, $\|\cdot\|_{\V'}$, and $\|\cdot\|_{\V''}$, be norms on $\V_0$, $\V'$, and $\V''$, respectively.

\begin{definition}\label{definition quasi hadamard}
Let $f:\V_f\to\V'$ be a map defined on a subset $\V_f$ of $\V$, and $\C_0$ be a subset of $\V_0$. Then $f$ is said to be quasi-Hadamard differentiable at $\theta \in \V_f$ tangentially to $\C_0\langle\V_0\rangle$ if there is some continuous map $D_{\theta;\C_0\langle\V_0\rangle}^{\qHad}f:\C_0\to\V'$ such that
\begin{eqnarray}\label{def eq for HD}
    \lim_{n\to\infty}\Big\| D_{\theta;\C_0\langle\V_0\rangle}^{\qHad}f\,(v)-\frac{f(\theta+h_nv_n)-f(\theta)}{h_n}\Big\|_{\V'}\,=\,0
\end{eqnarray}
holds for each triplet $(v,(v_n),(h_n))$, with $v\in\C_0$, $(v_n)\subset\V_0$ satisfying $\|v_n-v\|_{\V_0}\to 0$ as well as $(\theta+h_nv_n)\subset\V_f$, and $(h_n)\subset(0,\infty)$ satisfying $h_n\to 0$. In this case the mapping $D_{\theta;\C_0\langle\V_0\rangle}^{\qHad}f$ is called quasi-Hadamard derivative of $f$ at $\theta$ tangentially to $\C_0\langle\V_0\rangle$.
\end{definition}

Notice that quasi-Hadamard differentiability tangentially to $\C_0\langle\V_0\rangle$ clearly implies quasi-Hadamard differentiability tangentially to $\B_0\langle\V_0\rangle$ for every $\B_0\subset\C_0$. In this case, $D_{\theta;\B_0\langle\V_0\rangle}^{\qHad}f=D_{\theta;\C_0\langle\V_0\rangle}^{\qHad}f|_{\B_0}$. Also notice that if $\|\cdot\|_{\V_0}$ provides a norm on all of $\V$, $\C_0=\V_0$, and the derivative is linear, then the notion of quasi-Hadamard differentiability at $\theta \in \V_f$ tangentially to $\C_0\langle\V_0\rangle$ coincides with the classical notion of Hadamard differentiability tangentially to $\V_0$ as defined in \cite{Gill1989}, and we write $D_{\theta;\V_0}^{\Had}f$ in place of $D_{\theta;\V_0\langle\V_0\rangle}^{\qHad}f$. We stress the fact that in general $D_{\theta;\V_0}^{\Had}f$ is not the same as $D_{\theta;\V_0\langle\V_0\rangle}^{\qHad}f$, because in the latter case the norm $\|\cdot\|_{\V_0}$ may only be defined on $\V_0$. The preceding discussion shows in particular that quasi-Hadamard differentiability is a weaker notion of ``differentiability'' than the classical (tangential) Hadamard differentiability in the sense of \cite{Gill1989}. However, it was shown in \cite{BeutnerZaehle2010} that this notion is still strong enough to obtain a generalized version of the Functional Delta-Method as given in Theorem 3 of \cite{Gill1989}.

The following chain rule can be proven in the same way as the chain rule in \cite[Theorem 20.9]{vanderVaart1998}; we omit the details. In Condition (b) we will {\em not} insist on linearity of the Hadamard derivative. Note that Hadamard differentiability with possibly nonlinear derivative has been studied before; see, for instance, \cite{Roemisch2006}.

\begin{lemma}\label{lemma chain rule}
Let $f:\V_f\to\V'$ be a map defined on a subset $\V_f$ of $\V$, and $\C_0$ be a subset of $\V_0$. Let $g:\V_g\to\V^{''}$ be a map defined on a subset $\V_g$ of $\V'$ with $f(\V_f) \subset \V_g$. Let $\V_0'$ be a subset of $\V'$, and assume that the following assertions hold:
\begin{itemize}
    \item[(a)] The map $f$ is quasi-Hadamard differentiable at $\theta \in \V_f$ tangentially to $\C_0\langle\V_0\rangle$ with quasi-Hadamard derivative $D_{\theta;\C_0\langle\V_0\rangle}^{\qHad}f$ satisfying $D_{\theta;\C_0\langle\V_0\rangle}^{\qHad}f\,(\C_0) \subset \V_0'$.
    \item[(b)] The map $g$ is Hadamard differentiable at $f(\theta)$ tangentially to $\V_0'$ with Hadamard derivative $D_{f(\theta);\V_0'}^{\Had}g$.
\end{itemize}
Then $g \circ f: \V_f\to\V{''}$ is quasi-Hadamard differentiable at $\theta$ tangentially to $\C_0\langle\V_0\rangle$ with quasi-Hadamard derivative $D_{\theta;\C_0\langle\V_0\rangle}^{\qHad}g\circ f=D_{f(\theta);\V_0'}^{\Had}g \circ D_{\theta;\C_0\langle\V_0\rangle}^{\qHad}f$.
\end{lemma}

\begin{definition}\label{definition quasi continuity}
Let $f:\V_f\to\V'$ be a map defined on a subset $\V_f$ of $\V$. The map $f$ is said to be quasi-Lipschitz continuous at $\theta\in\V_f$ along $\V_0$ if
\begin{eqnarray}\label{def eq for QC - def}
    \|f(\theta+u_n)-f(\theta)\|_{\V'}\,=\,{\cal O}(\|u_n\|_{\V_0})
\end{eqnarray}
holds for every sequences $(u_n)\subset\V_0\setminus\{0_\V\}$ with $(\theta+u_n)\subset\V_f$ and $\|u_n\|_{\V_0}\to 0$.
\end{definition}

\begin{lemma}\label{equiv QC}
Let $f:\V_f\to\V'$ be a map defined on a subset $\V_f$ of $\V$. Then the map $f$ is quasi-Lipschitz continuous at $\theta\in\V_f$ along $\V_0$ if and only if
\begin{eqnarray}\label{def eq for QC}
    \|f(\theta+h_nv_n)-f(\theta)\|_{\V'}\,=\,o(h_n)
\end{eqnarray}
holds for every sequences $(v_n)\subset\V_0$ and $(h_n)\subset(0,\infty)$ with $(\theta+h_nv_n)\subset\V_f$, $\|v_n\|_{\V_0}\to 0$ and $h_n\to 0$.
\end{lemma}

\begin{proof}
The sufficiency of (\ref{def eq for QC - def}) for (\ref{def eq for QC}) is obvious. In order to prove the necessity, let $f$ satisfy (\ref{def eq for QC}) and suppose that $f$ does {\em not} satisfy (\ref{def eq for QC - def}). Then there would exist a sequence $(u_n)\subset\V_0\setminus\{0_\V\}$ with $(\theta+u_n)\subset\V_f$ and $\|u_n\|_{\V_0}\to 0$, and a subsequence $(u_{n_k})\subset(u_n)$ such that
$$
    z_k\,:=\,\|f(\theta+u_{n_k})-f(\theta)\|_{\V'}\,/\,\|u_{n_k}\|_{\V_0}\,\rightarrow\,\infty,\qquad k\to\infty.
$$
To verify a contradiction, we set $h_k:=\|u_{n_k}\|_{\V_0}\,z_k$ and $v_k:=u_{n_k}/(\|u_{n_k}\|_{\V_0}z_k)$, where we assume without loss of generality $\|u_{n_k}\|_{\V_0}
\in (0,1]$ and $z_k>0$ for all $k\in\N$. Then, on one hand, we have $\theta+h_kv_k\,(=\theta+u_{n_k})\in\V_f$ and $\|v_k\|_{\V_0}\to 0$. On the other hand,
we have
\begin{eqnarray}
    h_k
    & = & \|u_{n_k}\|_{\V_0}\,z_k\nonumber\\
    & = & \|f(\theta+u_{n_k})-f(\theta)\|_{\V'}\nonumber\\
    & \le & \big\|f\big(\theta+\|u_{n_k}\|_{\V_0}^{1/2}\big\{u_{n_k}/\|u_{n_k}\|_{\V_0}^{1/2}\big\}\big)-f(\theta)\big\|_{\V'}\,/\,\|u_{n_k}\|_{\V_0}^{1/2}.\label{equiv QC - proof - 10}
\end{eqnarray}
Since $\theta+\|u_{n_k}\|_{\V_0}^{1/2}\{u_{n_k}/\|u_{n_k}\|_{\V_0}^{1/2}\}$ $(=\theta+u_{n_k})\in\V_f$, $\|u_{n_k}/\|u_{n_k}\|_{\V_0}^{1/2}\|_{\V_0}\to 0$ and $\|u_{n_k}\|_{\V_0}^{1/2}\to 0$, we may conclude from (\ref{equiv QC - proof - 10}) and (\ref{def eq for QC}) that $h_k\to 0$. Thus, in view of $0 < h_{k} = \|f(\theta+u_{n_k})-f(\theta)\|_{\V'}$ for every $k\in\N,$ we obtain by (\ref{def eq for QC})
$$
    1\,=\,\lim_{k\to\infty}\|f(\theta+h_kv_k)-f(\theta)\|_{\V'}\,/\,h_k\,=\,0.
$$
This is a contradiction.
\end{proof}

It is an immediate consequence of Lemma \ref{equiv QC} that quasi-Lipschitz continuity of $f$ at $\theta$ along $\V_0$ exactly coincides with quasi-Hadamard differentiability of $f$ at $\theta$ tangentially to $\{0_\V\}\langle\V_0\rangle$ with quasi-Hadamard derivative $D_{\theta;\{0_\V\}\langle\V_0\rangle}^{\qHad}f(0_\V)=0_{\V'}$, where $0_\V$ and $0_{\V'}$ denote the nulls in $\V$ and $\V'$, respectively. So we immediately obtain the following lemma.

\begin{lemma}\label{HC implies HD}
Let $f:\V_f\to\V'$ be a map defined on a subset $\V_f$ of $\V$, and $\C_0$ be a subset of $\V_0$ with $0_\V\in\C_0$. If $f$ is quasi-Hadamard differentiable at $\theta \in \V_f$ tangentially to $\C_0\langle\V_0\rangle$ with quasi-Hadamard derivative satisfying $D_{\theta;\C_0\langle\V_0\rangle}^{\qHad}f(0_\V)=0_{\V'}$, then $f$ is quasi-Lipschitz continuous at $\theta$ along $\V_0$. 
\end{lemma}


\section{Separability of the uniform metric on spaces of \cadlag\ functions}\label{Separability of uniform metric}

Let $D[0,1]$ denote the space of \cadlag\ functions on $[0,1]$. It is endowed with the uniform metric $d_{\infty}$. For any sequence $\underline{a} := (a_{k})\subset[0,1]$, we consider the set $D_{\underline{a}}[0,1]$ of all $v\in D[0,1]$ whose discontinuity points belong to $\{a_{k}: k\in\N\}$. We want to show that any such set is $d_{\infty}$-separable. Firstly we shall focus on the sets $D_{\underline{a}}$ based on finite sequences $\underline{a}.$

\begin{lemma}\label{finite separability}
The space $D_{\underline{a}}[0,1]$ is $d_{\infty}$-separable if $\underline{a}$ is a finite sequence in $[0,1]$.
\end{lemma}

\begin{proof}
By assumption, there exists $0 = a_{0} <\dots< a_{r+1} = 1$ such that $D_{\underline{a}}[0,1]$ consists of all $v\in D[0,1]$ whose restriction to $[0,1]\setminus\{a_{1}\dots a_{r + 1}\}$ are continuous. Let $C[a_{i},a_{i+1}]$ denote the space of all continuous real-valued mappings on $[a_{i},a_{i+1}]$ for $i\in\{0,\dots,r\}$. Then the mapping
$$
    \R\,\times\Timesir C[a_{i},a_{i+1}]\longrightarrow D_{\underline{a}}[0,1],\qquad
    (x,f_{0},\dots f_{r})\longmapsto \sum_{i=0}^{r} f_{i}\eins_{[a_{i},a_{i+1})} + x\eins_{\{1\}}
$$
is surjective and  continuous w.r.t.\ the metrics $d$ and $d_\infty$, where the metric $d$ on $\R\,\times\!\!\!\Timesir  C[a_{i},a_{i+1}]$ is defined by
$$
    d((x,f_{0},\dots,f_{r}),(y,g_{0},\dots,g_{r}))\,:=\, |x-y|\vee  \max_{i\in\{0,\dots,r\}}\,\sup_{t\in [a_{i},a_{i+1}]}|f_{i}(t) - g_{i}(t)|.
$$
The proof is complete because the metric $d$ is separable.
\end{proof}

\begin{lemma}\label{fixierte Sprunghoehe}
Let $\varepsilon > 0$ and $v\in D[0,1]$ be such that $|v(x) - v(x_{-})|\leq\varepsilon$ for every $x\in [0,1].$ Then there exists some continuous mapping $w: [0,1]\rightarrow\R$ satisfying $d_{\infty}(v,w) \leq 2\varepsilon.$
\end{lemma}

\begin{proof}
Let ${\cal W}_{v}$ denote the modulus of continuity of $v$, i.e.\ the mapping
$$
    {\cal W}_{v}: (0,1]\longrightarrow\R,\quad \delta\longmapsto \sup_{|x - y|\leq\delta}\, |v(x) - v(y)|.
$$
Since we have assumed that $|v(x) - v(x_{-})|\leq\varepsilon$ holds for any $x\in [0,1]$, Lemma 12.1 (with $\varepsilon/2$ in place of $\varepsilon$) and (12.9) in \cite{Billingsley1999} ensure that we may find some $\delta_{0}\in (0,1)$ such that ${\cal W}_{v}(\delta_{0}) \leq 2\varepsilon$. Moreover, consider the following mapping
$$
    \eta: \R\longrightarrow \R,\quad x\longmapsto
    \bcswitch
    C\exp(- \frac{1}{1 - x^{2}}) & : & |x| < 1\\
    0 & : & |x| \geq 1
    \ecswitch,
$$
where the positive constant $C$ is chosen so that $\int_{-1}^{1}\eta(x)\, dx = 1.$ Extending $v$ to a mapping $\widehat{v}$ on $\R$ by
$$
    \widehat{v}(x)\,:=\,
    \bcswitch
    v(0) & : & x\in [-\delta_{0},0)\\
    v(x) & : & x\in [0,1]\\
    v(1) & : & x\in (1,1+\delta_{0}]\\
    0 & : & \mbox{otherwise}
    \ecswitch,
$$
we obtain $\int_{\R}|\widehat{v}(x)|\,dx < \infty$. Then it is already known from \cite[p.\,630, Theorem 6]{Evans1998} that the mapping
$$
    \widehat{v}^{\,\delta_{0}}: (-2 +\delta_{0},2-\delta_{0})\longrightarrow\R,\quad x\longmapsto \int_{-\delta_{0}}^{\delta_{0}} \frac{1}{\delta_{0}}\,\eta\Big(\frac{z}{\delta_{0}}\Big)\,\widehat{v}(x - z)\, dz
$$
is infinitely differentiable. In particular, its restriction $w := \widehat{v}^{\,\delta_{0}}|_{[0,1]}$ to $[0,1]$ is continuous. Moreover, for any $x\in [0,1]$, we have
\begin{equation}\label{Annaeherung}
    |w(x) - v(x)|\,=\,\Big|\int_{-1}^{1}\eta(z)\widehat{v}(x - \delta_{0} z)\, dz - \widehat{v}(x)\Big|
    \,\le\,\int_{-1}^{1}\eta(z)|\widehat{v}(x - \delta_{0} z) - \widehat{v}(x)|\, dz.
\end{equation}
On the one hand, if $z\in [-1,1]$ with $x - \delta_{0} z < 0$, then $x < \delta_{0}$ and
$$
    |\widehat{v}(x - \delta_{0} z) - \widehat{v}(x)|\,=\,|v(0) - v(x)|\,\le\,{\cal W}_{v}(\delta_{0})\,\le\, 2\varepsilon.
$$
On the other hand, if $z\in [-1,1]$ with $x - \delta_{0} z > 1$, then
$1 - x < \delta_{0} $ and
$$
    |\widehat{v}(x - \delta_{0} z) - \widehat{v}(x)|\,=\,|v(1) - v(x)|\le {\cal W}_{v}(\delta_{0})\,\le\, 2\varepsilon.
$$
Finally, if $z\in [-1,1]$ with $x - \delta_{0} z \in [0,1]$, then $|x -\delta_{0} z - x| \leq \delta_{0}$ and
$$
    |\widehat{v}(x - \delta_{0} z) - \widehat{v}(x)|\,=\,|v(x - \delta_{0} z) - v(x)|\,\le\, {\cal W}_{v}(\delta_{0})\leq 2\varepsilon.
$$
Hence we may conclude from (\ref{Annaeherung})
$$
    |w(x) - v(x)|\,\leq\,2\varepsilon \int_{-1}^{1}\eta(z)\,dz\,=\,2 \varepsilon.
$$
This completes the proof since $x\in [0,1]$ was arbitrarily chosen.
\end{proof}

\begin{theorem}\label{general separability}
The space $D_{\underline{a}}[0,1]$ is $d_{\infty}$-separable for every sequence $\underline{a} = (a_{k})$.
\end{theorem}

\begin{proof}
Let $\underline{a}_{k} := \{a_{1},\ldots,a_{k}\}$ for $k\in\N$. In view of Lemma \ref{finite separability}, it remains to show that for every $\varepsilon > 0,$ and any $v\in D_{\underline{a}}[0,1]$, there exist $k_{0}\in\N$ and some $w\in D_{\underline{a}_{k_{0}}}[0,1]$ such that $d_{\infty}(v,w) \leq \varepsilon$. For that purpose let us fix $\varepsilon > 0$ as well as $v\in D_{\underline{a}}[0,1]$. It is well known that the set $\{x\in [0,1]:|v(x) - v(x_{-})| > \varepsilon/2\}$ is finite. Hence there are $k_{1},\ldots,k_{r}\in\N$ and $0=: t_{0} < t_{1} <\cdots < t_{r}\leq t_{r+1} := 1$ such that $t_{i}= a_{k_{i}}$ for $i = 1,\dots,r,$ and $|v(x) - v(x_{-})| \leq \varepsilon/2$ for $x\in [0,1)\setminus\{t_{1},\dots,t_{r}\}$. If $1 = a_{k}$ for some $k\in\N,$ we choose $t_{k_{r}} = a_{k_{r}} = 1$. Select any $k_{0}\in\N$ such that $\{a_{k_{1}},\ldots,a_{k_{r}}\}\subset\{a_{k}: k\in\N\mbox{ with }k\leq k_{0}\}$.

Next, let us define for $i\in\{0,\dots,r\}$ the mapping
$$
    v_{i}: [t_{i},t_{i+1}]\longrightarrow\R,\quad x\longmapsto
    \bcswitch
    v(x) & : &  x < t_{i+1}\\
    v(t_{i+1}-) & : &  x =t_{i+1}.
    \ecswitch
$$
Obviously, $v_{i}$ is a c\`adl\`ag function on $[t_{i},t_{i+1}]$ with $|v_{i}(x) - v_{i}(x_{-})|\leq \varepsilon/2$ for $x\in[t_i,t_{i+1}]$. Since any interval $[a,b]$ with
$a < b$ is homeomorphic with $[0,1]$ via some strictly increasing mapping, we may find by Lemma \ref{fixierte Sprunghoehe} some continuous mapping $w_{i}: [t_{i},t_{i+1}]\rightarrow\R$ satisfying
$|w_{i}(x) - v_{i}(x)|\leq\varepsilon$ for every $x\in [t_{i},t_{i+1}].$ Then the function
$$
    w\,:=\,\sum\limits_{i=0}^{r} w_{i}\eins_{[t_{i},t_{i+1})} + \Big(w_{r}(1)\eins_{\{0\}}(v(1_{-}) - v(1)) + v(1)\eins_{\R\setminus\{0\}}(v(1_{-}) - v(1)) \Big)\eins_{\{1\}}
$$
belongs to $D_{\underline{a}_{k_{0}}}[0,1]$ fulfilling $d_{\infty}(v,w)\leq\varepsilon.$
This completes the proof.
\end{proof}

Let $C_{u,\phi,F_{0}}$ and $C_{0,\phi,F_{0}}$ denote the sets of all $v\in C_{\phi,F_{0}}$ satisfying $\lim_{x\to\pm\infty}v(x)\phi(x) = c_\pm$ for any constants $c_-,c_+\in\R$ and $\lim_{x\to\pm\infty}v(x)\phi(x) = 0$, respectively.

\begin{corollary}\label{monotone separable}
The sets $C_{u,\phi,F_{0}}$ and $C_{0,\phi,F_{0}}$ are $\|\cdot\|_{\phi}$-separable and ${\cal D}_{\phi,F_0}$-measurable.
\end{corollary}

\begin{proof}
Let $\underline{b} = (b_{k})$ denote the (possibly finite) sequence of discontinuities of $F_0$, viewed as a subset of the compactification $[-\infty,\infty]$ of $\R$. We may and do pick a strictly increasing homeomorphism $\iota$ from $[F_{0}^{\rightarrow}(0),F_{0}^{\leftarrow}(1)]$ onto $[0,1]$. Let $D_{\iota(\underline{b})}[0,1]$ be the set of all \cadlag\ functions on $[0,1]$ whose discontinuity points belong to $\{\iota(b_{k}):k\in\N\}$. The mapping $\gamma: D_{\iota(\underline{b})}[0,1]\rightarrow C_{u,\phi,F_0}$ defined by
$$
    x\longmapsto\gamma(g)(x)\,:=\,
    \left\{
        \begin{array}{c@{\quad :\quad}l}
            \frac{g\circ\iota(x)}{\phi(x)} & x\in [F_{0}^{\rightarrow}(0),F_{0}^{\leftarrow}(1)]\\
            0 & \mbox{otherwise}
        \end{array}
    \right.
$$
is surjective and continuous w.r.t.\ the uniform metric on $D_{\iota(\underline{b})}[0,1]$ and $\|\cdot\|_{\phi}$. Since the uniform metric on $D_{\iota(\underline{b})}[0,1]$ is separable by Theorem \ref{general separability}, we may conclude that $C_{u,\phi,F_0}$ is separable w.r.t.\ $\|\cdot\|_{\phi}$. The same arguments show that $C_{0,\phi,F_0}$ is separable w.r.t.\ $\|\cdot\|_{\phi}$. Finally notice that $C_{u,\phi,F_0}$ and $C_{0,\phi,F_0}$ are closed subsets of $D_{\phi,F_{0}}$ w.r.t.\ $\|\cdot\|_{\phi}$. This implies that these sets belong to ${\cal D}_{\phi,F_0}$; cf.\ \cite[hint for Problem 1.7.4]{vanderVaartWellner1996}. This completes the proof.
\end{proof}

\bibliography{QHDbib}{}
\bibliographystyle{abbrv}
\end{document}